\documentclass[10pt]{article}

\usepackage{enumerate}

\usepackage{tikz}
   \def\tikzscale{0.8}
\usepackage{amsmath, amsfonts, amssymb}
\usepackage{latexsym}
\usepackage{amsthm}
\usepackage{bbm}
\usepackage{bbm}
\usepackage[T1]{fontenc}
\usepackage{graphicx}
\usepackage{textcomp}

\usepackage[vmargin={1.25in,1.25in},
            hmargin={1.3in,1.3in}
            ]{geometry}

\usepackage{floatrow}
\newfloatcommand{capbtabbox}{table}[][\FBwidth]

\usepackage{url}

\def\Rec{\mathop{{\rm Rec}}\nolimits}
\def\dist{\mathop{{\rm dist}}\nolimits}
\def\diam{\operatorname{diam}}
\def\NN{{\mathbb N}}
\def\PP{{\mathbb P}}

\def\TT{{\mathbb T}}
\def\ZZ{{\mathbb Z}}

\def\({\left (}
\def\){\right )}

\newcommand{\floor}[1]{\left\lfloor #1 \right\rfloor}
\newcommand{\ceil}[1]{\left\lceil #1 \right\rceil}
\newcommand{\closure}[1]{\left\langle #1 \right\rangle}

\newtheorem{theorem}{Theorem}
\newtheorem{lemma}[theorem]{Lemma}

\newtheorem{prop}[theorem]{Proposition}
\newtheorem{question}[theorem]{Question}
\newtheorem{conj}[theorem]{Conjecture}

\newtheorem{fact}[theorem]{Fact}
\newtheorem{observation}[theorem]{Observation}
\newtheorem{defn}[theorem]{Definition}
\newtheorem{claim}[theorem]{Claim}
\newtheorem{remark}[theorem]{Remark}

\begin{document}

\title{Maximum percolation time in two-dimensional bootstrap percolation}

\author{Fabricio Benevides\thanks{Departamento de Matem{\'a}tica, Universidade Federal do Cear{\'a}, Av. Humberto Monte, s/n., Bloco~914, Fortaleza, CE 60455-760, Brazil, E-mail: fabricio@mat.ufc.br. Supported by FUNCAP and CNPq.} \and Micha{\l} Przykucki\thanks{Department of Pure Mathematics and Mathematical Statistics, University of Cambridge, Wilberforce Road, Cambridge CB3 0WB, UK, and London Institute for Mathematical Sciences, 35a South St, Mayfair, London W1K 2XF, UK. Supported in part by MULTIPLEX no.\ 317532. E-mail: michal@london-institute.org}}

\maketitle

\begin{abstract}
We consider a classic model known as bootstrap percolation on the $n \times n$ square grid. To each vertex of the grid we assign an initial state, infected or healthy, and then in consecutive rounds we infect every healthy vertex that has at least $2$ already infected neighbours. We say that percolation occurs if the whole grid is eventually infected. In this paper, contributing to a recent series of extremal results in this field, we prove that the maximum time a bootstrap percolation process can take to eventually infect the entire vertex set of the grid is $13n^2/18+O(n)$.
\end{abstract}

\section{Introduction}
\label{sec:intro}
In this paper we consider a particular extremal problem in $2$-neighbour bootstrap percolation on the $n \times n$ square grid. Our aim is to find the maximum time it may take to infect the whole grid, as precisely defined below.

Given a graph $G$ and a natural number $r$, consider the following process known as $r$-neighbour bootstrap percolation on $G$. Choose a subset $A \subset V (G)$ of vertices (that in the context of percolation are usually called sites) and infect all of its elements, leaving the remaining vertices healthy. Then, in consecutive rounds, infect every healthy site that has at least $r$ already infected neighbours. More formally, set $A_0 = A$ and, for $t \in \NN$, thinking of $A_t$ as the set of sites that have been infected by time $t$ and denoting by $N(v)$ the set of neighbours of $v$, let 
\begin{equation}
\label{eq:A_t}
A_{t} = A_{t-1} \cup \{ v \in V(G) : |N(v)\cap A_{t-1}| \geq r \}.
\end{equation}
The set of all sites that eventually become infected is called the \emph{closure} of $A$ and is denoted by $\closure{A}$. It is clear from the definition that $\closure{A} = \bigcup_{t=0}^{\infty} A_t$. We say that a set $A$ \emph{percolates} if all sites are eventually infected, that is, if $\closure{A} = V(G)$. The choice of the set $A$ may be either random or deterministic, giving rise to questions of different nature. 

Bootstrap percolation was introduced in 1979 by Chalupa, Leath and Reich~\cite{bootstrapbethe} and has found applications in many areas including physics, computer science and sociology. One of the first questions that attracted a lot of attention was related to the critical probability defined as
\[
 p_c (G , r) = \inf \{p : \PP_p (A \text{ percolates in } G \text{ in } r \text{-neighbour bootstrap process}) \geq 1/2 \},
\]
where the elements of the set $A$ are chosen independently at random with probability~$p$. In the most classical and celebrated variant the graph $G$ is the $n \times n$ square grid, denoted by $[n]^2$ (i.e., the set of sites is $V(G) = \{(i,j): 1 \le i, j \le n\}$ and two sites are adjacent if they are at $l_1$ distance 1) and $r=2$. Working in this setup Aizenman and Lebowitz~\cite{metastabilityeffects} showed that  $p_c ([n]^2 , 2) = \Theta \( \frac{1}{\log n} \)$. Using much more sophisticated techniques Holroyd \cite{sharpmetastability} showed that $p_c ([n]^2 , 2) = \frac{\pi^2}{18\log n} + o\(\frac{1}{\log n}\)$, and Gravner, Holroyd and Morris \cite{sharperThreshold} obtained bounds on the second order term. Cerf and Cirillo \cite{scalingthree} and Cerf and Manzo \cite{regimefinite} determined the critical probability for $r$-neighbour bootstrap percolation on $[n]^d$ up to a constant factor, while recently Balogh, Bollob{\'a}s, Duminil-Copin and Morris \cite{sharpbootstrapall} obtained the 
asymptotic value of $p_c ([n]^d , r)$ for all fixed values of $d$ and $r$. More general models of bootstrap percolation were studied by Gravner and Griffeath \cite{thresholdGrowth}, Bollob\'as, Smith and Uzzell \cite{neighbourhoodBootstrap} and Balister, Bollob\'as, Przykucki and Smith \cite{subcriticalBootstrap}.

Now, when the set $A$ is chosen deterministically, various interesting extremal questions arise. The first observation one can make is now a folklore one: in the classic model, with $G=[n]^2$ and $r=2$, the smallest percolating sets have size exactly~$n$. The size of the smallest percolating sets in other graphs and for other values of the infection threshold was studied by Pete \cite{randomdisease} and by Balogh, Bollob{\'a}s, Morris and Riordan \cite{linearAlgebra}. Answering a question posed by Bollob{\'a}s, Morris~\cite{largestgridbootstrap} gave bounds on the maximum size of a minimal percolating set for $G = [n]^2$ and $r=2$. A similar problem for 2-neighbour bootstrap percolation on a hypercube was fully answered by Riedl~\cite{largesthypercubebootstrap} who also studied minimal percolating sets in finite trees \cite{minimalTrees}. In this paper we continue this recent trend and consider another extremal problem posed by Bollob{\'a}s. We give an asymptotic value of the maximum time that any 
percolating subset of the set of vertices of $G = [n]^2$ can take to percolate under $2$-neighbour bootstrap percolation. The main result of this article is the following theorem.
\begin{theorem}
 \label{theorem:arbitraryAsymptotics}
 The maximum time of percolation on the $n \times n$ square grid is $\frac{13}{18}n^2+O(n)$.
\end{theorem}

An analogous question for a hypercube was recently answered by Przykucki \cite{przykucki01}. In \cite{benevidesprzykucki02}, Benevides and Przykucki showed that, again for $G = [n]^2$ and $r=2$, when we restrict our attention to percolating sets of size $n$ then the maximum percolation time is equal to the integer nearest to $\frac{5n^2-2n}{8}$. Together with Theorem~\ref{theorem:arbitraryAsymptotics} this implies that, somewhat surprisingly, the slowest percolating sets do not have the minimum possible number of sites.

Benevides, Campos, Dourado, Sampaio and Silva \cite{maxtimegeneralgraphs} considered the computational complexity of the question of finding maximum percolation time on general graphs. They proved that its associated decision problem is NP-complete. Questions related to percolation time have also been considered recently in the probabilistic setup by Bollob{\'a}s, Holmgren, Smith and Uzzell~\cite{densebootstraptime} and by Bollob{\'a}s, Smith and Uzzell~\cite{densebootstraptimeall}.

The structure of this paper is as follows. In Section~\ref{sec:preliminaries} we introduce the basic notation and define $M(k,\ell)$, the function representing the maximum percolation time on the $k \times \ell$ grid. In Section~\ref{sec:recurrency} we define a particular family of percolating sets, prove the asymptotic formula for $M(k,\ell)$ and show that our family contains sets that percolate in time $M(k,\ell)$. In Section~\ref{sec:valueMnn} we prove Theorem~\ref{theorem:arbitraryAsymptotics} and finally in Section~\ref{sec:questions} we show some results that follow from our work and state some open questions and conjectures.

\subsection{Relationship to earlier work}
We acknowledge that some of the techniques used in~\cite{benevidesprzykucki02} are also applied here. However, the condition that $|A|=n$ that was assumed in~\cite{benevidesprzykucki02} makes the question much easier to answer and greatly simplifies the proof. In the case of arbitrary percolating sets our simulations suggest that the maximum percolation time in $[n]^2$ is obtained for sets of size $23n/18 + O(1)$, implying that these two questions are significantly different.

Let us briefly outline here the additional complications that arise when we consider the problem for arbitrary percolating sets. A reader not familiar with our previous article~\cite{benevidesprzykucki02} may skip the remainder of this section. Throughout this paper we shall also explicitly point out the new key ideas in our proofs.

First, the condition $|A|=n$ imposes very strong limitations on the rectangles $R'$ and $R''$ when we apply Proposition \ref{prop:rectangles}. The lack of these limitations makes the definition of $(k,\ell)$-perfect sets in Section \ref{sec:recurrency} much wider than the definition of $(k,\ell)$-good sets in \cite{benevidesprzykucki02}. Consequently, in the proof of the lower bound on $M(k,\ell)$ in Theorem~\ref{theorem:Mnrecurrence} we need to take into the account four constructions of percolating sets that could not occur when $|A| = n$. More importantly, in the proof of the upper bound in Theorem~\ref{theorem:Mnrecurrence} we need to consider three additional situations, namely Condition~\ref{item:CondD},~\ref{item:CondE} and most crucially~\ref{item:CondF}, that could occur when we apply Proposition \ref{prop:rectangles}. The analysis of Condition~\ref{item:CondF} is by far the most important ingredient of the proof of Theorem~\ref{theorem:Mnrecurrence}.

The first part of the proof of Theorem~\ref{theorem:arbitraryAsymptotics}, which we give in Section \ref{sec:valueMnn}, is a significant extension of the methods used in \cite{benevidesprzykucki02}. However, the most important element of the proof is the application of the \emph{fractional moves} that are a new idea introduced in this paper. We explain the precise reasons behind the need to study this new concept in Section~\ref{sec:valueMnn}.

\section{Notation and preliminary observations}
\label{sec:preliminaries}
We write $\Rec(k, \ell)$ to denote the set of all $k$ by $\ell$ rectangles in $\ZZ^2$, i.e., of all subsets of the integer lattice of the form $\{a, a+1, \ldots, a+k-1\} \times \{b, b+1, \ldots, b+\ell-1\}$ for some $a,b \in \ZZ$. When we represent subsets of $\ZZ^2$ graphically we depict $(i,j) \in \ZZ^2$ as a unit square centred at $(i,j)$. We usually use shaded squares to mark infected sites.

The \emph{perimeter} of a set $A \subset \ZZ^2$ is the number of edges between $A$ and $\ZZ^2 \setminus A$ in the integer lattice graph. In our applications it will be more convenient to talk about $\Phi(A)$, the \emph{semi-perimeter} of $A$, which is simply half of its perimeter. Thus, for $R \in \Rec(k, \ell)$ we have $\Phi(R) = k+\ell$.

When we talk about a \emph{distance} between two sites in $\ZZ^2$ we always mean the usual graph distance, i.e., the length of the shortest path between two vertices, that for sites $(i_1,j_1)$, $(i_2,j_2) \in \ZZ^2$ is equal to $|i_1 - i_2| + |j_1 - j_2|$. For two subsets $A, B$ of $\ZZ^2$ the \textit{distance} between them, $\dist(A, B)$, is the minimum distance between a site in $A$ and a site in $B$. Clearly, $\dist(A, B) = 0$ if and only if $A \cap B \neq \emptyset$ and $\dist(A, B) = 1$ if their intersection is empty but there is a site in $A$ that is adjacent to a site in $B$. In our pictures two such sites correspond to unit squares that share an edge. 

Now let us turn to $2$-neighbour bootstrap percolation on the integer lattice. A rectangle $R$ is said to be \emph{internally spanned} by a set $A$ of infected sites if $\closure{A\cap R} = R$. Let us observe that for any set $A$ of initially infected sites we have $\Phi(\closure{A}) \le \Phi(A)$. This is because whenever a new site becomes infected at least two edges are removed from the boundary of the infected set and at most two new edges are added to it. Also, every edge can transmit infection only once from a uniquely determined infected site to a uniquely determined healthy site. Thus the perimeter of the infected area cannot grow during the process. From this observation we have the following fact.

\begin{fact} 
\label{fact:n_necessary}
Given $R \in \Rec(k,\ell)$, if $A \subset R$ internally spans $R$ then $|A| \geq \ceil{\Phi(R)/2} = \ceil{\frac{k+\ell}{2}}$. In particular, if $n \in \NN$ and $A \subset [n]^2$ percolates, then $|A| \ge n$.
\end{fact}

Another simple observation is that, for any set $A$ of infected sites, $\closure{A}$ is a union of rectangles such that any distinct two of them are at distance at least $3$. This can be observed immediately as $A$ is, indeed, a union of $1 \times 1$ rectangles and any two rectangles at distance at most $2$ internally span the minimal rectangle containing them both.

The next proposition from Holroyd~\cite{sharpmetastability}, giving us a deeper insight into the nature of percolating sets, shall be extremely useful in our further considerations.
\begin{prop}\label{prop:rectangles}
Let $R$ be a rectangle with area at least $2$ internally spanned by a set $A$. Then there exist disjoint subsets $A', A'' \subsetneq A$, and subrectangles $R', R'' \subsetneq R$ such that:
\begin{enumerate}
\item \label{prop:rectangles:c1} $\closure{A'} = R'$ and $\closure{A''} = R''$, and
\item \label{prop:rectangles:c2} $\closure{R'\cup R''} = R$; in particular, $\dist(R', R'') \le 2$.
\end{enumerate}
\end{prop}

In Proposition~\ref{prop:rectangles} we cannot require $R' \cap R'' = \emptyset$ (see Figure~\ref{figure:overlap}). Also, the choices of $A'$ and $A''$ (and hence also of $R'$ and $R''$) are not necessarily unique. Furthermore, we note that given a set $A$ some sites in $R \setminus (R'\cup R'')$ may become infected in the process while some of $R'\cup R''$ are still healthy. Finally, the disjointness of $A'$ and $A''$ is often a crucial ingredient when we try to answer probabilistic questions in bootstrap percolation. However, it will not be important for our purposes.

\begin{figure}[ht]
  \centering
  \begin{tikzpicture}[scale=\tikzscale]
    \fill[color=lightgray] (0,3.5) rectangle +(0.5,0.5);
    \fill[color=lightgray] (0.5,3) rectangle +(0.5,0.5);
    \fill[color=lightgray] (1,2.5) rectangle +(0.5,0.5);
    \fill[color=lightgray] (1.5,5) rectangle +(0.5,0.5);
    \fill[color=lightgray] (2,4.5) rectangle +(0.5,0.5);
    \fill[color=lightgray] (2.5,4) rectangle +(0.5,0.5);
    \fill[color=lightgray] (2.5,1) rectangle +(0.5,0.5);
    \fill[color=lightgray] (3,0.5) rectangle +(0.5,0.5);
    \fill[color=lightgray] (3.5,0) rectangle +(0.5,0.5);
    \fill[color=lightgray] (4,2.5) rectangle +(0.5,0.5);
    \fill[color=lightgray] (4.5,2) rectangle +(0.5,0.5);
    \fill[color=lightgray] (5,1.5) rectangle +(0.5,0.5);

    \draw[step=.5cm,gray,thick]
    (0,0) grid (5.5,5.5);
    \draw[line width=1mm,black,dashed]
    (0,5.5) -- (3,5.5) -- (3,2.5) -- (0,2.5) -- (0,5.5);
    \draw[line width=1mm,black,dashed]
    (2.5,3) -- (5.5,3) -- (5.5,0) -- (2.5,0) -- (2.5,3);
    \draw[black] (-0.25,4) node [left] {$R'$};
    \draw[black] (5.75,1.5) node [right] {$R''$};
  \end{tikzpicture}
  \caption{An example where the overlapping rectangles $R'$ and $R''$ are uniquely determined by the initially infected sites.}
  \label{figure:overlap}
\end{figure}
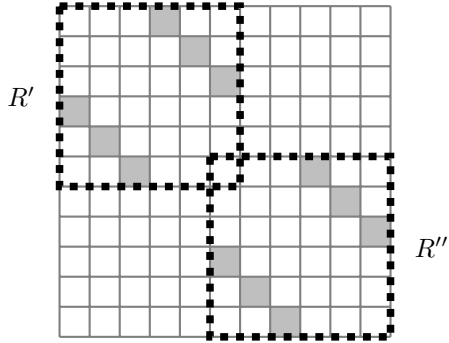

Now let us define the notion of maximum percolation time precisely. For a graph $G$ and a set $A$ of initially infected sites we say that $A$ takes time $T$ to percolate (or ``percolates in time $T$'') if $\closure{A} = V(G)$ and $T$ is the smallest number such that $A_T = V(G)$, where $A_t$ is defined as in \eqref{eq:A_t}. We shall also be interested in the infection time of particular sites $v \in V(G)$. Therefore let $I_A(v)$ be the minimum $T$ such that $v \in A_T$ starting from $A_0 = A$. If starting from $A$ the site $v$ never becomes infected, i.e., if $v \notin \closure{A}$, then we set $I_A(v) = \infty$. Finally, we define
\[
M(n) = \max \{T \in \NN: \mbox{there exists a set } A \mbox{ percolating in time } T \mbox{ in } [n]^2 \} .
\]
In this paper we determine the asymptotic formula for $M(n)$ up to an $O(n)$ additive error. We believe that a constant additive error or maybe even an exact formula could be found with similar techniques but with a much longer and tedious proof. We show that to infect $[n]^2$ in the maximum possible time one should first infect some smaller rectangular grid, not necessarily a square one, in the maximum time. This motivates a definition of the maximum percolation time in rectangles. For any $k, \ell \in \NN$ let
\[
M(k,\ell) = \max \{T \in \NN: \mbox{there exists a set } A \mbox{ percolating in time } T \mbox{ in } [k] \times [\ell] \} .
\]
Note that clearly $M(k,\ell) = M(\ell,k)$. For a rectangle $R \in \Rec(k,\ell)$, to simplify our notation, we shall often write $M(R)$ instead of $M(k,\ell)$.

\section{Slowly percolating sets}
\label{sec:recurrency}

In this section we prove the recursive formula for $M(k,\ell)$ in order to later prove the asymptotic formula for $M(n)$. Let us start by giving a trivial upper bound and a natural lower bound on $M(n)$. Since every percolating set in $[n]^2$ contains at least $n$ sites and for the infection to continue at every step we need to infect at least one new site, we have $M(n) \leq n^2-n$. On the other hand, the example shown in Figure~\ref{figure:2n^2/3} for the $[7]^2$ grid, generalizing in a self--explanatory way to $[n]^2$, shows that there exist initially infected sets of size linear in $n$ for which at approximately half of the number of steps only one site becomes infected while the other steps, with the exception of the first one, yield infection of only two new sites. This clearly implies that $M(n) \geq \frac{2n^2}{3}+O(n)$. We will prove that for every $n$ there is a set which percolates $[n]^2$ in time $M(n)$, for which at every time step at most two new sites become infected, but the number of steps 
for which a single site becomes infected is significantly larger than in the example in Figure~\ref{figure:2n^2/3}.

\begin{figure}[ht]
  \centering
    \begin{tikzpicture}[scale=\tikzscale]
      \foreach \x/\y in {0/0, 1/0, 2/0, 3/0, 0/1, 3/1.5, 0/2.5, 3/3}
	\fill[color=lightgray]  (\x,\y) rectangle +(0.5,0.5);
      \draw[black] (0.75,0.25) node {1};
      \draw[black] (1.75,0.25) node {1};
      \draw[black] (2.75,0.25) node {1};

      \draw[step=.5cm,gray,thick]
      (0,0) grid (3.5,3.5);
      \draw[->] (0.25,0.75) -- (3.25,0.75);
      \draw[->] (0.75,0.75) -- (0.75,1.25) -- (3.25,1.25);
      \draw[->] (2.75,1.25) -- (2.75,1.75) -- (0.25,1.75) -- (0.25,2.25) -- (3.25,2.25);
      \draw[->] (0.75,2.25) -- (0.75,2.75) -- (3.25,2.75);
      \draw[->] (2.75,2.75) -- (2.75,3.25) -- (0.25,3.25);
    \end{tikzpicture}
  \caption{An initial set giving a lower bound $M(n) \geq \frac{2n^2}{3}+O(n)$.}
  \label{figure:2n^2/3}
\end{figure}

The outline of our proof is as follows. First we define a notion of a \textit{$(k,\ell)$-perfect} set of initially infected sites. Next, we prove that the function $M(k,\ell)$ satisfies a certain recursive relation. Simultaneously we show that $(k,\ell)$-perfect sets exist and that their percolation time satisfies the same relation as does the function $M(k,\ell)$. Although we do not find an exact solution to the recursion, we are able to find good lower and upper bounds on $M(n)$. For the lower bound we construct an explicit set of initially infected sites that is ``almost'' $(n,n)$-perfect. Finally, for the upper bound, we define a relaxed version of the infection process and for any $(n,n)$-perfect set $A$ we build an appropriate instance of this new process; from this new instance we get an upper bound for the time that $A$ takes to percolate. Most of the important ideas necessary to obtain an upper bound on $M(n)$ are new and have not appeared in the previous works related to maximum percolation time.

The overall structure of this paper is similar to the one of \cite{benevidesprzykucki02} where we defined the notion of \textit{$(k,\ell)$-good} sets. However, even though it might not seem immediately obvious, the notion of a $(k,\ell)$-perfect set is not only a technical improvement over the $(k,\ell)$-good sets. The freedom arising from the ability to choose an arbitrary number of initially infected sites greatly increases the variety of percolating sets, forcing us to ``beat'' all of them when it comes to percolation time using sets that we have good control over. These shall be precisely our $(k,\ell)$-perfect sets.

Let us start introducing the $(k,\ell)$-perfect sets now. The idea is to look at sets of initially infected sites, say $A$, for which the infection process started from $A$ can be described by a nested sequence of rectangles $P_0 \subset P_1 \subset \ldots \subset P_r$, such that for every $0 \leq i \leq r$ the set $A \cap P_i$ internally spans $P_i$ in maximum time. We shall only consider sequences such that $P_0$ is small and $\Phi(P_{i-1}) + 2 \leq \Phi(P_{i}) \leq \Phi(P_{i-1}) + 3$. That gives us seven possible values of the lengths of the sides of $P_i$ given those of $P_{i-1}$ (see Figures \ref{figure:Moves123} and \ref{figure:Moves4567}). We should note that it is far from obvious that such a set $A$ exists. We prove this in Theorem \ref{theorem:Mnrecurrence} by induction. In order to make this induction easier we add a few other technical conditions to the definition of a $(k,\ell)$-perfect set. Now let us define this notion precisely.

\begin{defn}\label{def:perfectset}
Given $k,\ell \in \NN$ we say that a set $A$ is $(k,\ell)$-perfect if the infection process starting from $A$ can be described in the following way. There exists a nested sequence of rectangles $P_0 \subset P_1 \subset \ldots \subset P_r \in \Rec(k,\ell)$ with $P_i \in \Rec(s_i, t_i)$ for every $0\le i \le r$, satisfying the following properties:
\begin{enumerate} \renewcommand{\theenumi}{\alph{enumi}} 
  \item \label{calRmax:a} either $s_0 \le 2$ or $t_0 \le 2$ or $s_0 = t_0 = 3$; and $s_1, t_1 \geq 3$ with $(s_1, t_1) \neq (3,3)$,
  \item \label{calRmax:b} for each $1 \le i \le r$,
\[
\begin{split}
 P_i & \in \Rec(s_{i-1}+1,t_{i-1}+1) \cup \Rec(s_{i-1}+2,t_{i-1}) \cup \Rec(s_{i-1},t_{i-1}+2) \\
	     & \cup \Rec(s_{i-1}+2,t_{i-1}+1) \cup \Rec(s_{i-1}+1,t_{i-1}+2) \\
             & \cup \Rec(s_{i-1},t_{i-1}+3) \cup \Rec(s_{i-1}+3,t_{i-1}),
\end{split}
\]
  \item \label{calRmax:c} for every $0\le i \le r$, the rectangle $P_i$ is internally spanned by $A \cap P_i$ in the maximum possible time, that is, in time $M(P_i)$,
  \item \label{calRmax:d} for every $0\le i \le r$, if $P_i$ has no side of length $1$ then among the sites becoming infected last in $P_i$ there is at least one of its corner sites,
  \item \label{calRmax:e} for every $1\le i \le r$, if
   \[
    P_i \in \Rec(s_{i-1}+1,t_{i-1}+1) \cup \Rec(s_{i-1},t_{i-1}+2) \cup \Rec(s_{i-1}+2,t_{i-1})
   \]
then there exists a site $v_{i-1} \in A$ such that $P_{i-1} \cup \{v_{i-1}\}$ internally spans $P_{i}$ and $v_{i-1}$ is at distance exactly 2 from one of the corner sites in $P_{i-1}$ (one which becomes infected last in $P_{i-1}$, if there is such) and at distance at least 3 from any other site in $P_{i-1}$ (see Figure~\ref{figure:Moves123}),
  \item \label{calRmax:f} for every $1\le i \le r$, if 
\[
\begin{split}
    P_i & \in \Rec(s_{i-1}+2,t_{i-1}+1) \cup \Rec(s_{i-1}+1,t_{i-1}+2) \\
	    & \cup \Rec(s_{i-1},t_{i-1}+3) \cup \Rec(s_{i-1}+3,t_{i-1})
\end{split}
\]
then there exists a pair of sites $v_{i-1}, w_{i-1} \in A$ such that $P_{i-1} \cup \{v_{i-1}, w_{i-1}\}$ internally spans $P_{i}$ and $v_{i-1}$ is at distance exactly 2 from one of the corner sites in $P_{i-1}$ (one which becomes infected last in $P_{i-1}$, if there is such) and at distance at least 3 from any other site in $P_{i-1}$, while $w_{i-1}$ is at distance exactly 1 from one of the last corner sites to become infected in $\closure{P_{i-1} \cup \{v_{i-1}\}}$ and at distance at least 2 from any other site in $\closure{P_{i-1} \cup \{v_{i-1}\}}$ (see Figure~\ref{figure:Moves4567}).
\end{enumerate}
\end{defn}

From condition (\ref{calRmax:c}), taking $i = r$, it follows that any $(k,\ell)$-perfect set infects a rectangle in $\Rec(k,\ell)$ in time $M(k,\ell)$. In particular, any $(n,n)$-perfect set maximizes percolation time in $[n]^2$.

Given a $(k,\ell)$-perfect set and a sequence $P_0 \subset P_1 \subset \ldots \subset P_r \in \Rec(k,\ell)$ associated with it, for $1 \le i \le r$ and $1\le m \le 7$, we say that we use \emph{Move~$m$} at moment $i$ (to construct $P_{i}$ from $P_{i-1}$) if $P_{i}$ belongs to the $m$-th term of the following list:
\begin{enumerate}
 \item $\Rec(s_{i-1}+1,t_{i-1}+1)$,
 \item $\Rec(s_{i-1}+2,t_{i-1})$,
 \item $\Rec(s_{i-1},t_{i-1}+2)$,
 \item $\Rec(s_{i-1}+2,t_{i-1}+1)$,
 \item $\Rec(s_{i-1}+1,t_{i-1}+2)$,
 \item $\Rec(s_{i-1},t_{i-1}+3)$,
 \item $\Rec(s_{i-1}+3,t_{i-1})$.
\end{enumerate}

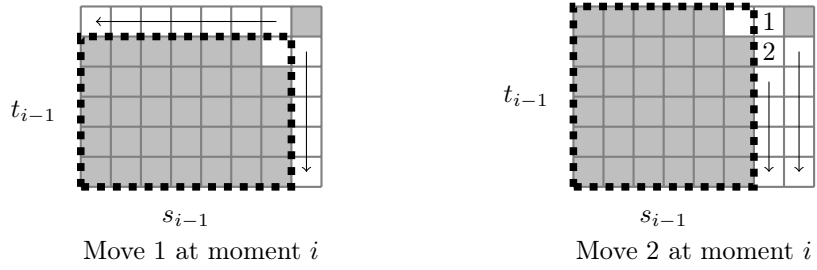
\begin{figure}[ht]
  \centering
  \begin{minipage}[b]{0.4\linewidth}
    \centering
    \begin{tikzpicture}[scale=\tikzscale]
      \fill[color=lightgray]  (3.5,2.5) rectangle +(0.5,0.5);
      \fill[color=lightgray]  (0,0) rectangle (3.5,2);
      \fill[color=lightgray]  (0,2) rectangle (3,2.5);

      \draw[step=.5cm,gray,thick]
      (0,0) grid (4,3);
      \draw[line width=1mm,black,dashed]
      (0,0) -- (0,2.5) -- (3.5,2.5) -- (3.5,0) -- (0,0);
      \draw[->] (3.75,2.25) -- (3.75,0.25);
      \draw[->] (3.25,2.75) -- (0.25,2.75);
      \draw[black] (-0.25,1.25) node [left] {$t_{i-1}$};
      \draw[black] (1.75, -0.85) node [above] {$s_{i-1}$};
      \draw[black] (2,-0.75) node [below] {Move $1$ at moment $i$};
    \end{tikzpicture}
  \end{minipage}
  \hspace{0.5cm}
  \begin{minipage}[b]{0.4\linewidth}
    \centering
    \begin{tikzpicture}[scale=\tikzscale]
      \fill[color=lightgray]  (3.5,2.5) rectangle +(0.5,0.5);
      \fill[color=lightgray]  (0,0) rectangle (3,2.5);
      \fill[color=lightgray]  (0,2.5) rectangle (2.5,3);

      \draw[step=.5cm,gray,thick]
      (0,0) grid (4,3);
      \draw[line width=1mm,black,dashed]
      (0,0) -- (0,3) -- (3,3) -- (3,0) -- (0,0);
      \draw[->] (3.75,2.25) -- (3.75,0.25);
      \draw[->] (3.25,1.75) -- (3.25,0.25);
      \draw[black] (-0.25,1.5) node [left] {$t_{i-1}$};
      \draw[black] (4.25,1.5) node [right] {};
      \draw[black] (1.5,-0.85) node [above] {$s_{i-1}$};
      \draw[black] (3.25,2.75) node {1};
      \draw[black] (3.25,2.25) node {2};
      \draw[black] (2,-0.75) node [below] {Move $2$ at moment $i$};
    \end{tikzpicture}
  \end{minipage}
  \caption{Move $1$ and $2$ (Move $3$ is obtained by rotating the picture of Move $2$ by 90 degrees).}
  \label{figure:Moves123}
\end{figure}

\begin{figure}[ht]
  \centering
  \begin{minipage}[b]{0.4\linewidth}
    \centering
    \begin{tikzpicture}[scale=\tikzscale]
      \fill[color=lightgray]  (3.5,2.5) rectangle +(0.5,0.5);
      \fill[color=lightgray]  (0,2.5) rectangle +(0.5,0.5);
      \fill[color=lightgray]  (0.5,0) rectangle (3.5,2);
      \fill[color=lightgray]  (0,0) rectangle (0.5,1.5);

      \draw[step=.5cm,gray,thick]
      (0,0) grid (4,3);
      \draw[line width=1mm,black,dashed]
      (0,0) -- (0,2) -- (3.5,2) -- (3.5,0) -- (0,0);
      \draw[->] (0.75,2.75) -- (3.25,2.75);
      \draw[->] (1.25,2.25) -- (3.75,2.25) -- (3.75,0.25);
      \draw[black] (0.25,2.25) node {1};
      \draw[black] (0.75,2.25) node {2};
      \draw[black] (-0.25,1) node [left] {$t_{i-1}$};
      \draw[black] (4.25,1.5) node [right] {};
      \draw[black] (1.75, -0.85) node [above] {$s_{i-1}$};
      \draw[black] (2,-0.75) node [below] {Move $5$ at moment $i$};
    \end{tikzpicture}
  \end{minipage}
  \hspace{0.5cm}
  \begin{minipage}[b]{0.4\linewidth}
    \centering
    \begin{tikzpicture}[scale=\tikzscale]
      \fill[color=lightgray]  (3.5,2.5) rectangle +(0.5,0.5);
      \fill[color=lightgray]  (3,0) rectangle +(0.5,0.5);
      \fill[color=lightgray]  (0,0) rectangle (2,3);
      \fill[color=lightgray]  (2,0.5) rectangle (2.5,3);

      \draw[step=.5cm,gray,thick]
      (0,0) grid (4,3);
      \draw[line width=1mm,black,dashed]
      (0,0) -- (0,3) -- (2.5,3) -- (2.5,0) -- (0,0);
      \draw[->] (3.75,2.25) -- (3.75,0.25);
      \draw[->] (2.75,1.25) -- (2.75,2.75);
      \draw[->] (3.25,0.75) -- (3.25,2.25);
      \draw[black] (-0.25,1.5) node [left] {$t_{i-1}$};
      \draw[black] (4.25,1.5) node [right] {};
      \draw[black] (1.25,-0.85) node [above] {$s_{i-1}$};
      \draw[black] (2.75,0.25) node {1};
      \draw[black] (2.75,0.75) node {2};
      \draw[black] (2,-0.75) node [below] {Move $7$ at moment $i$};
    \end{tikzpicture}
  \end{minipage}
    \caption{Move $5$ and $7$ (Move $4$ and $6$ are obtained by rotating the above figures by 90 degrees).}
    \label{figure:Moves4567}
\end{figure}
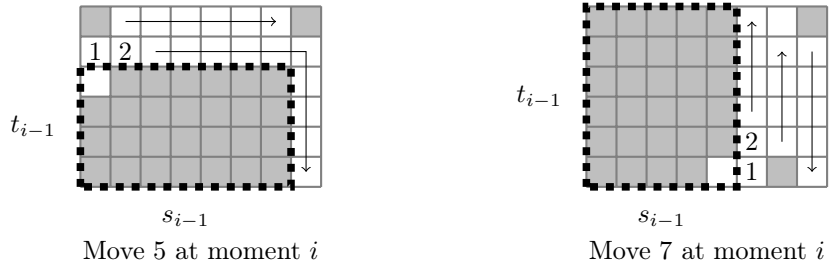

In the next lemma we determine the value of $M(k,2)$ and give an example of a $(k,2)$-perfect set for each natural $k$.

\begin{lemma}\label{lemma:kequals2}
For any natural number $k$ we have $M(k,2) = \floor{\frac{3(k-1)}{2}}$. Furthermore, there is a $(k,2)$-perfect set $A^0(k,2)$ that percolates $[k] \times [2]$ in time $M(k,2)$.
\end{lemma}
\begin{proof}
First let us consider the case when $k$ is even. Let $A^0(k,2)$ to be the set of shaded sites in Figure~\ref{figure:2k}. Clearly $A^0(k,2)$ percolates $[k] \times [2]$ in time $\floor{\frac{3(k-1)}{2}} = (3k-4)/2$. Thus we have $M(k, 2) \ge \floor{\frac{3(k-1)}{2}}$ for any $k$ even.

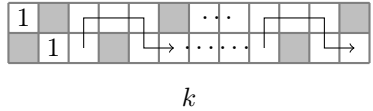
\begin{figure}[ht]
  \centering
  \begin{tikzpicture}[scale=\tikzscale]
  
    \fill[color=lightgray]  (0,0) rectangle +(0.5,0.5);
    \fill[color=lightgray]  (0.5,0.5) rectangle (1,1);
    \fill[color=lightgray]  (1.5,0) rectangle (2,0.5);
    \fill[color=lightgray]  (2.5,0.5) rectangle (3,1);
    \fill[color=lightgray]  (4.5,0) rectangle (5,0.5);
    \fill[color=lightgray]  (5.5,0.5) rectangle (6,1);

    \draw[step=.5cm,gray,thick]
    (0,0) grid (6,1);

    \draw[->] (1.25,0.25) -- (1.25,0.75) -- (1.75,0.75) -- (2.25,0.75) -- (2.25,0.25) -- (2.75,0.25);
    \draw[->] (4.25,0.25) -- (4.25,0.75) -- (4.75,0.75) -- (5.25,0.75) -- (5.25,0.25) -- (5.75,0.25);
    \draw[black] (0.25,0.75) node {1};
    \draw[black] (0.75,0.25) node {1};
    \draw[black] (3.5,0.75) node {\ldots};
    \draw[black] (3.5,0.25) node {\ldots\ldots};
    \draw[black] (4.25,1.5) node [right] {};
    \draw[black] (3,-0.25) node [below] {$k$};
  \end{tikzpicture}
  \caption{A $(k,2)$-perfect set achieving maximum percolation time on $[k] \times [2]$ for $k$ even.}
  \label{figure:2k}
\end{figure}

Now we prove by induction on $k$ that for any $k$ even we have $M(k,2) \le (3k-4)/2$. Clearly, ${M(2, 2) = 1}$. Assume that $k \ge 4$ is even and that $M(k-2,2) = (3k-10)/2$. Let $A$ be any set that percolates $[k] \times [2]$. Since percolation time is at most the number of initially healthy sites, if $|A| \ge k/2+2$ then it percolates in time at most $2k-(k/2+2) = (3k-4)/2$. On the other hand, by Fact \ref{fact:n_necessary}, we must have $|A| \ge k/2+1$. Therefore we may assume that the cardinality of $A$ is exactly $k/2+1$.

Since $A$ percolates, for all $1\le i \le k/2$, any $2 \times 2$ square of the form $\{2i-1, 2i\}\times\{1,2\}$ contains at least one site of $A$. Hence only one of these squares contains two sites of $A$. Therefore, either $\{1,2\} \times \{1, 2\}$ or $\{k-1, k\} \times \{1,2\}$ contains exactly one such site. Assume without loss of generality that the latter holds. As $A$ percolates, either $(k,1)$ or $(k,2)$ must be an initially infected site. Again without loss of generality we may assume that the latter holds. In this setting it is trivial to check that $A \setminus \{(k,2)\}$ internally spans $[k-2]\times [2]$. Therefore $A$ takes time at most $M(k-2,2) + 3 = (3k-4)/2$ to percolate. It is also trivial to check that the sequence consisting of only one rectangle, say $P_0 = [k] \times [2]$, satisfies the conditions in Definition~\ref{def:perfectset}.

For $k$ odd, the set in Figure~\ref{figure:2kOdd} has the minimum cardinality necessary for a set to percolate $[k] \times [2]$ and at each time step causes infection of only one site. Therefore it percolates in the maximum time that is indeed $\floor{\frac{3(k-1)}{2}}$. It is an immediate observation that it satisfies all conditions of a $(k,2)$-perfect set.
\end{proof}

\begin{figure}[ht]
  \centering
  \begin{tikzpicture}[scale=\tikzscale]
  
    \fill[color=lightgray]  (0,0) rectangle +(0.5,1);
    \fill[color=lightgray]  (1,0) rectangle +(0.5,0.5);
    \fill[color=lightgray]  (2,0.5) rectangle +(0.5,0.5);
    \fill[color=lightgray]  (4,0) rectangle +(0.5,0.5);
    \fill[color=lightgray]  (5,0.5) rectangle +(0.5,0.5);

    \draw[step=.5cm,gray,thick]
    (0,0) grid (5.5,1);

    \draw[->] (0.75,0.25) --(0.75,0.75) -- (1.75,0.75) -- (1.75,0.25) -- (2.25,0.25);
    \draw[->] (3.75,0.25) -- (3.75,0.75) -- (4.75,0.75) -- (4.75,0.25) -- (5.25,0.25);
    \draw[black] (3,0.75) node {\ldots};
    \draw[black] (3,0.25) node {\ldots\ldots};
    \draw[black] (3.75,1.5) node [right] {};
    \draw[black] (2.75,-0.25) node [below] {$k$};
  \end{tikzpicture}
  \caption{A $(k,2)$-perfect set achieving maximum percolation time on $[k] \times [2]$ for $k$ odd.}
  \label{figure:2kOdd}
\end{figure}
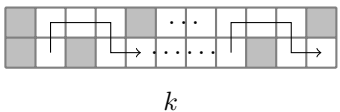

In the next theorem we state the recursive formula for $M(k, \ell)$. We should keep in mind the description of $(k,\ell)$-perfect initially infected sets because the proof of the theorem is built on the proof of existence and a construction of such sets. Since $M(k, \ell) = M(\ell, k)$, we shall omit some cases where $k < \ell$.

\begin{theorem}
\label{theorem:Mnrecurrence}
We have $M(1,1) = M(2,1) = 0$; $M(k, 1) = 1$ for all $k \ge 3$; $M(k, 2) = \floor{\frac{3(k-1)}{2}}$; and $M(3, 3) = 4$. For $k, \ell \ge 3$ such that $(k, \ell) \neq (3,3)$, we have
\begin{equation}
\label{equation:Mnrecurrence}
M(k,\ell) = \max
  \begin{cases}
    M(k-1,\ell-1) + \max\{k, \ell\}-1,\\
    M(k-2, \ell) + \ell+1,\\
    M(k,\ell-2) + k+1,\\
    M(k-2,\ell-1)+k+\ell-2,\\
    M(k-1,\ell-2)+k+\ell-2,\\
    M(k,\ell-3)+2k-1,\\
    M(k-3,\ell)+2\ell-1,
  \end{cases}
\end{equation}
where we assume $M(k,0) = M(0,\ell) = - \infty$. Furthermore, for any $k,\ell>0$ there exists a $(k, \ell)$-perfect set.
\end{theorem}
\begin{proof}
We prove Theorem~\ref{theorem:Mnrecurrence} by induction on $k + \ell$. A small case analysis immediately gives the result for $\ell=1$ and for $(k,\ell) = (3,3)$. For $\ell = 2$ we use Lemma~\ref{lemma:kequals2}. Note that in all these cases there exist $(k, \ell)$-perfect initial sets for which, in Definition~\ref{def:perfectset} (of $(k,\ell)$-perfect sets), we have $r=0$. 

Now, assume that we are given $k, \ell \ge 3$ such that $(k, \ell) \neq (3,3)$. Our induction hypothesis is that for any $k', \ell' \ge 1$ such that $k'+\ell' < k+\ell$ there exists a $(k', \ell')$-perfect set $A^M(k', \ell')$ that percolates in time $M(k', \ell')$, as in the statement of Theorem~\ref{theorem:Mnrecurrence}. We continue the proof as follows. First, we show that $M(k,\ell)$ is at least the right-hand-side of equation~\eqref{equation:Mnrecurrence}. We do this by presenting seven constructions of sets percolating the $[k] \times [\ell]$ grid in times corresponding to the terms on the right-hand-side of~\eqref{equation:Mnrecurrence}. We use our induction hypothesis to estimate the percolation times of our constructions and to observe that these constructions are ``nearly'' $(k, \ell)$-perfect, i.e., that they satisfy all properties of $(k, \ell)$-perfect sets apart from, possibly, infecting a $[k] \times [\ell]$ grid in time $M(k,\ell)$.

We then move on to bounding $M(k,\ell)$ from above. Given any set $A$ percolating $[k] \times [\ell]$ we use Proposition~\ref{prop:rectangles} to show that $A$ must satisfy one of the six particular Conditions that impose upper bounds on the time that $A$ takes to percolate. We analyse these Conditions one by one, with the last one of them (i.e., with Condition~\ref{item:CondF}) requiring the most detailed analysis, and deduce that the right-hand-side of~\eqref{equation:Mnrecurrence} is an upper bound on $M(k,\ell)$. This implies that $M(k,\ell)$ indeed satisfies equation~\eqref{equation:Mnrecurrence} and that $(k, \ell)$-perfect sets exist for all values of $k$ and $\ell$.

Let $W(k,\ell)$ denote the right-hand-side of equation~\eqref{equation:Mnrecurrence}. With this notation, we want to show that $M(k,\ell) = W(k,\ell)$. We shall first prove that
\begin{equation}
\label{equation:Mngeq}
M(k,\ell) \geq W(k,\ell).
\end{equation}

Assume without loss of generality that $k \geq 4$. Recall that for $k', \ell' \ge 2$ by the definition of $(k', \ell')$-perfect sets we may assume that one of the corners of the rectangle spanned by $A^M(k', \ell')$ becomes infected at time $M(k', \ell')$. Now, consider the following seven ways of infecting $[k] \times [\ell]$ (see again Figures~\ref{figure:Moves123}~and~\ref{figure:Moves4567}), some of which we define only for slightly larger values of $k$ and $\ell$.

\begin{enumerate}
  \item \label{case:M1} Let $\closure{A^M(k-1, \ell-1)} = [k-1]\times [\ell-1]$. Since $k-1, \ell -1 \ge 2$ we may assume that $(k-1,\ell-1)$ becomes infected at time $M(k-1,\ell-1)$. Let $A^{(1)} = A^M(k-1, \ell -1) \cup \{(k, \ell)\}$. Then $A^{(1)}$ takes time $M(k-1,\ell-1) + \max\{k,\ell\} - 1$ to percolate.
  
  \item \label{case:M3} Let $\closure{A^M(k-2, \ell)} = [k-2]\times [\ell]$. Since $k-2, \ell \ge 2$ we may assume that $(k-2,\ell)$ becomes infected at time $M(k-2,\ell)$. Let $A^{(2)} = A^M(k-2, \ell) \cup \{(k, \ell)\}$. Then $A^{(2)}$ takes time $M(k-2,\ell) +\ell+1$ to percolate.
  
  \item \label{case:M2} For $\ell \ge 4$ we have $k, \ell-2 \ge 2$. Let $\closure{A^M(k, \ell-2)} = [k]\times [\ell-2]$. We may assume that $(k, \ell-2)$ becomes infected at time $M(k, \ell-2)$. Let $A^{(3)} = A^M(k, \ell-2) \cup \{(k, \ell)\}$. Then $A^{(3)}$ percolates in time $M(k,\ell-2) + k + 1$.
  
  \item \label{case:M5} Let $\closure{A^M(k-2, \ell-1)} = [k-2]\times [\ell-1]$. Since $k-2, \ell-1 \ge 2$ we may assume that $(k-2,1)$ becomes infected at time $M(k-2,\ell-1)$. Let $A^{(4)} = A^M(k-2, \ell-1) \cup \{(k, 1),(k, \ell)\}$. Then $A^{(4)}$ takes time $M(k-2,\ell-1)+k+\ell-2$ to percolate.
  
  \item \label{case:M4} For $\ell \ge 4$ we have $k-1, \ell-2 \ge 2$. Let $\closure{A^M(k-1, \ell -2)} = [k-1]\times [\ell-2]$. We may assume that $(1,\ell-2)$ becomes infected at time $M(k-1,\ell-2)$. Let $A^{(5)} = A^M(k-1, \ell -2) \cup \{(1, \ell),(k, \ell)\}$. Then $A^{(5)}$ takes time $M(k-1,\ell-2)+k +\ell - 2$ to percolate.
  
  \item \label{case:M6} For $\ell \geq 5$ we have $k, \ell-3 \geq 2$. Let $\closure{A^M(k, \ell-3)} = [k]\times [\ell-3]$ and assume that $(k,\ell-3)$ becomes infected at time $M(k,\ell-3)$. Let $A^{(6)} = A^M(k, \ell-3) \cup \{(k, \ell-1), (1, \ell)\}$. Then $A^{(6)}$ percolates in time $M(k,\ell-3)+2k -1$.
  
  \item \label{case:M7} For $k \geq 5$ an analogous construction to case (\ref{case:M6}) with a $(k-3,\ell)$-perfect set $A^M(k-3, \ell)$ spanning $[k-3]\times [\ell]$ in time $M(k-3, \ell)$. Taking $A^{(7)} = A^M(k-3, \ell)  \cup \{(k-1, \ell),(k, 1)\}$ we obtain a set spanning $[k] \times [\ell]$ in time $M(k-3,\ell)+2\ell-1$.
\end{enumerate}

The above constructions show that inequality~\eqref{equation:Mngeq} holds when $k, \ell \ge 5$. We now check that inequality~\eqref{equation:Mngeq} also holds for the small values of $k$ and $\ell$ for which some of these constructions cannot be applied. Constructions (\ref{case:M2}) and (\ref{case:M4}) do not apply when $\ell=3$ because we cannot ask for one of the corners of the smaller rectangles to become infected respectively at times $M(k, \ell-2) = 1$ and $M(k-1, \ell-2) = 1$. However, since $k \geq 4$, in these cases we have $M(k, \ell -2) +k +1 = k+2$ and $M(k-1, \ell -2) +k +\ell - 2 = k+2$ that is at most $M(k-1,\ell-1)+k-1 = \floor{\frac{3(k-2)}{2}} +k-1 \geq k + 2$.

Construction (\ref{case:M6}) does not apply for $\ell = 4$ since then we again cannot ask for one of the corners of $[k]\times [\ell-3]$ to become infected at time $M(k, \ell-3) = 1$. However, for $\ell=4$ we have $M(k, \ell-3) + 2k - 1 = 2k$ that is less than $M(k,\ell-2) + k + 1 = \floor{\frac{3(k-1)}{2}}+k+1 \geq \floor{\frac{2k+1}{2}}+k+1 = 2k+1$. Analogously we deal with the fact that construction (\ref{case:M7}) does not apply for $k = 4$. Thus the lower bound on $M(k,\ell)$ is proved.

For each of the sets $A^{(j)}$  constructed above, among the sites of $\closure{A^{(j)}}$ that become infected last there is a corner of $[k] \times [\ell]$. Thus it is clear that all sets $A^{(j)}$ satisfy properties (\ref{calRmax:a})-(\ref{calRmax:f}) of $(k,\ell)$-perfect sets except for, possibly, property (\ref{calRmax:c}). To finish the proof of Theorem~\ref{theorem:Mnrecurrence} we only need to prove the upper bound on $M(k,\ell)$ analogous to inequality \eqref{equation:Mngeq}. This will imply that at least one of the sets $A^{(j)}$ percolates in time $M(k, \ell)$ and therefore is $(k, \ell)$-perfect.

Hence, it remains to show that the right-hand-side of equation \eqref{equation:Mnrecurrence} is also an upper bound on $M(k,\ell)$, i.e., that
\begin{equation}
\label{equation:Mnleq}
M(k,\ell) \leq W(k,\ell).
\end{equation}


Let $A$ be any set that internally spans the rectangle $R = [k] \times [\ell]$ in time $M(k, \ell)$. Consider sets $A'$, $A''$ and rectangles $R'$, $R''$ satisfying conditions \ref{prop:rectangles:c1} and \ref{prop:rectangles:c2} of Proposition~\ref{prop:rectangles}. Define $T(R', R'')$ to be the time it takes to grow from $R' \cup R''$ to $R = \closure{R' \cup R''}$, that is, the time needed to infect all sites in $R \setminus (R' \cup R'')$ given that all sites in $R'$ and $R''$ are infected and no site in $R \setminus (R'\cup R'')$ is. Note that $T(R', R'')$ is a simple function that depends only on the dimensions of $R'$ and $R''$ and how they are positioned in $R$ but not on the underlying set $A$.
Let
\[
 S(R', R'') = \max\{M(R'), M(R'')\} + T(R', R'').
\]
Notice that $S(R', R'')$ also depends only on $R'$ and $R''$. It is clearly seen that for any choice of $A', A'' \subset A$ satisfying Proposition~\ref{prop:rectangles} we have that $S(R', R'')$ is an upper bound on the time that $A$ takes to percolate. We shall prove that, in order for $A$ to percolate in time $M(k, \ell)$, the rectangles $R'$ and $R''$ must be aligned according to (at least) one of the ways we describe below. In most cases a simple upper bound on $S(R', R'')$ will yield that $S(R', R'') \le W(k,\ell)$ and consequently that $M(k,\ell) \le W(k,\ell)$. However, in one particular case we might have $S(R', R'') > W(k,\ell)$. We will then need to be more careful and find an upper bound better than $S(R', R'')$ for the time that $A$ takes to percolate.

In the upcoming cases our technique of bounding $S(R', R'')$ will require the following claim saying that under our induction hypothesis the maximum percolation time is strictly increasing in the size of the underlying grid.

\begin{claim}
\label{claim:increaseby1}
Let $s$, $t$ be such that $s+t < k+\ell$. If $s \ge 1$ and $t \ge 2$ then $M(s+1,t) \ge M(s,t) + 1$. Similarly, if $s \geq 2$ and $t \ge 1$ then $M(s,t+1) \ge M(s,t) + 1$.
\end{claim}
{\par{\it Proof of Claim~\ref{claim:increaseby1}}. \ignorespaces}
Let $s \ge 1$ and $t \ge 2$. For $s = 1$, the result is trivial (as $M(2,2) \ge 1$ and $M(1,2) = 0$ and, for $t \geq 3$, $M(2,t) \ge 3$ and $M(1,t) = 1$). For $s, t \ge 2$, with $s+t < k +\ell$, by the induction hypothesis, we may assume that there exists a set $A^M(s,t)$ which internally spans the rectangle $[s] \times [t]$ in time $M(s,t)$ and, without loss of generality, such that
\[
 I_{A^M(s,t)}(s,t)=M(s,t) \geq 1.
\]
Note that we must have some $1 \leq i \leq t-1$ such that $(s, i) \in A^M(s,t)$. Let $i^*$ be the smallest such $i$. Let $\tilde{A} = A^M(s,t) \cup \{(s+1,i^*)\}$. Clearly $\closure{\tilde{A}} = [s+1] \times [t]$ and for any $j \in [t] \setminus \{i^*\}$ we have $I_{\tilde{A}}(s+1,j) \geq I_{A^M(s,t)}(s,j)+1$. Thus $M(s+1,t) \geq I_{\tilde{A}}(s+1,t) \geq M(s,t)+1$. \endproof

Assume without loss of generality that $M(R') \ge M(R'')$. Note that, in order to internally span $R$, the rectangles $R'$ and $R''$ must be at distance $0$, $1$ or $2$. Consider some minimal non-empty rectangle $\tilde{R}'' \subset R''$ such that $R' \cup \tilde{R}''$ spans $R$. Whenever $R'$ and $R''$ intersect, that is whenever $\dist(R',R'') = 0$, we can choose $\tilde{R}''$ so that it is disjoint from $R'$. Furthermore, whenever $\dist(R',R'') = 1$ then unless $R''$ has a side of length $1$ we can always choose $\tilde{R}''$ such that $\dist(R',\tilde{R}'')=2$. Since $T(R', R'') \le T(R', \tilde{R}'')$ and (by Claim \ref{claim:increaseby1}) $M(R') \ge M(R'') \geq M(\tilde{R}'')$, we have $S(R', R'') \le S(R', \tilde{R}'')$. Let $R' \in \Rec(s_1, t_1)$ and $\tilde{R}'' \in \Rec(s_2, t_2)$. With case analysis we find that, since $\tilde{R}''$ is chosen to be minimal, $R'$ and $\tilde{R}''$ must either satisfy one of the following conditions or their analogues obtained by swapping $k$ with $\ell$ (see 
Figure~\ref{figure:Cond}).

\begin{enumerate}[{Condition} A:]
  \item \label{item:CondB} rectangles $R'$ and $\tilde{R}''$ align as in Figure~\ref{figure:Cond}~(\ref{item:CondB}) with $s_1 + s_2 = k$, $t_1 + t_2 = \ell$.
  \item \label{item:CondA} rectangles $R'$ and $\tilde{R}''$ align as in Figure~\ref{figure:Cond}~(\ref{item:CondA}) with $s_1 + s_2 = k-1$ and $t_1 + t_2 = \ell+1$.
  \item \label{item:CondC} there is an $0 \leq m \leq \ell-1$ so that the rectangles $R'$ and $\tilde{R}''$ align as in Figure \ref{figure:Cond}~(\ref{item:CondC}) with $s_1 + s_2 = k-1$, $t_1 = \ell$ and $t_2 = 1$.
  \item \label{item:CondD} there is an $0 \leq m \leq \ell-t_1$ such that the rectangles $R'$ and $\tilde{R}''$ align as in Figure~\ref{figure:Cond} (\ref{item:CondD}) with $s_1 + s_2 = k-1$, $t_1<\ell$, $t_2=\ell$.
  \item \label{item:CondE} there is an $0 \leq m \leq \ell-t_1$ such that the rectangles $R'$ and $\tilde{R}''$ align as in Figure~\ref{figure:Cond} (\ref{item:CondE}) with $s_1 = k-1$, $s_2=1$, $t_1<\ell$, $t_2=\ell$.
  \item \label{item:CondF} there is an $0 \leq m \leq \ell-1$ such that the rectangles $R'$ and $\tilde{R}''$ align as in Figure \ref{figure:Cond} (\ref{item:CondF}) with $s_1 = k-1$, $s_2 = 1$, $t_1 = \ell$, $t_2 = 1$.
\end{enumerate}

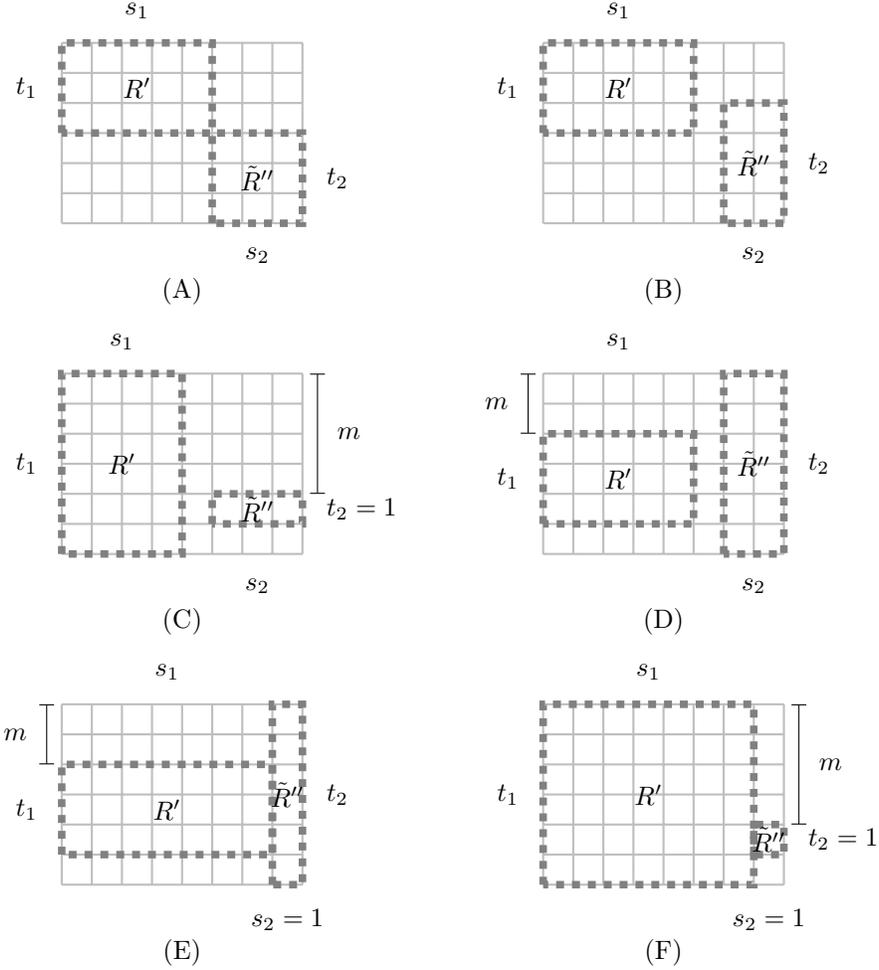
\begin{figure}[htb]
  \centering
    \begin{tikzpicture}[scale=\tikzscale]
  \begin{scope}[xshift=0cm,yshift=11cm]
      \draw[step=.5cm,lightgray,thick]
      (0,0) grid (4,3);
      \draw[line width=1mm,gray,dashed]
      (0,1.5) -- (0,3) -- (2.5,3) -- (2.5,1.5) -- (0,1.5);
      \draw[line width=1mm,gray,dashed]
      (2.5,0) -- (2.5,1.5) -- (4,1.5) -- (4,0) -- (2.5,0);
      \draw[black] (-0.25,2.25) node [left] {$t_1$};
      \draw[black] (4.25,0.75) node [right] {$t_2$};
      \draw[black] (1.25,3.25) node [above] {$s_1$};
      \draw[black] (3.25,-0.25) node [below] {$s_2$};
      \draw[black] (2,-0.75) node [below] {(\ref{item:CondB})};
      \draw[black] (1.25,2.25) node {$R'$};
      \draw[black] (3.25,0.75) node {$\tilde{R}''$};
  \end{scope}
  \begin{scope}[xshift=8cm,yshift=11cm]
      \draw[step=.5cm,lightgray,thick]
      (0,0) grid (4,3);
      \draw[line width=1mm,gray,dashed]
      (0,1.5) -- (0,3) -- (2.5,3) -- (2.5,1.5) -- (0,1.5);
      \draw[line width=1mm,gray,dashed]
      (3,0) -- (3,2) -- (4,2) -- (4,0) -- (3,0);
      \draw[black] (-0.25,2.25) node [left] {$t_1$};
      \draw[black] (4.25,1) node [right] {$t_2$};
      \draw[black] (1.25,3.25) node [above] {$s_1$};
      \draw[black] (3.5,-0.25) node [below] {$s_2$};
      \draw[black] (2,-0.75) node [below] {(\ref{item:CondA})};
      \draw[black] (1.25,2.25) node {$R'$};
      \draw[black] (3.5,1) node {$\tilde{R}''$};
  \end{scope}
  \begin{scope}[xshift=0cm,yshift=5.5cm]
    \draw[step=.5cm,lightgray,thick]
    (0,0) grid (4,3);
    \draw[line width=1mm,gray,dashed]
    (0,0) -- (0,3) -- (2,3) -- (2,0) -- (0,0);
    \draw[line width=1mm,gray,dashed]
    (2.5,0.5) -- (2.5,1) -- (4,1) -- (4,0.5) -- (2.5,0.5);
    \draw[black] (-0.25,1.5) node [left] {$t_1$};
    \draw[black] (4.25,0.75) node [right] {$t_2=1$};
    \draw [|-|, black] (4.25,1)  -- (4.25,3)
         node [black,midway,right=4pt] {$m$};
    \draw[black] (1,3.25) node [above] {$s_1$};
    \draw[black] (3.25,-0.25) node [below] {$s_2$};
    \draw[black] (2,-0.75) node [below] {(\ref{item:CondC})};
    \draw[black] (1,1.5) node {$R'$};
    \draw[black] (3.25,0.75) node {$\tilde{R}''$};
  \end{scope}
  \begin{scope}[xshift=8cm,yshift=5.5cm]
      \draw[step=.5cm,lightgray,thick]
      (0,0) grid (4,3);
      \draw[line width=1mm,gray,dashed]
      (0,0.5) -- (0,2) -- (2.5,2) -- (2.5,0.5) -- (0,0.5);
      \draw[line width=1mm,gray,dashed]
      (3,0) -- (3,3) -- (4,3) -- (4,0) -- (3,0);
      \draw[black] (-0.25,1.25) node [left] {$t_1$};
      \draw[black] (4.25,1.5) node [right] {$t_2$};
      \draw[black] (1.25,3.25) node [above] {$s_1$};
      \draw[black] (3.5,-0.25) node [below] {$s_2$};
    \draw[black] (2,-0.75) node [below] {(\ref{item:CondD})};
    \draw [|-|, black] (-0.25,2)  -- (-0.25,3)
       node [black,midway,left=4pt] {$m$};
    \draw[black] (1.25,1.25) node {$R'$};
    \draw[black] (3.5,1.5) node {$\tilde{R}''$};
  \end{scope}
  \begin{scope}[xshift=0cm,yshift=0cm]
      \draw[step=.5cm,lightgray,thick]
      (0,0) grid (4,3);
      \draw[line width=1mm,gray,dashed]
      (0,0.5) -- (0,2) -- (3.5,2) -- (3.5,0.5) -- (0,0.5);
      \draw[line width=1mm,gray,dashed]
      (3.5,0) -- (3.5,3) -- (4,3) -- (4,0) -- (3.5,0);
      \draw[black] (-0.25,1.25) node [left] {$t_1$};
      \draw[black] (4.25,1.5) node [right] {$t_2$};
      \draw[black] (1.75,3.25) node [above] {$s_1$};
      \draw[black] (3.75,-0.25) node [below] {$s_2=1$};
    \draw[black] (2,-0.75) node [below] {(\ref{item:CondE})};
    \draw [|-|, black] (-0.25,2)  -- (-0.25,3)
       node [black,midway,left=4pt] {$m$};
    \draw[black] (1.75,1.25) node {$R'$};
    \draw[black] (3.75,1.5) node {$\tilde{R}''$};
  \end{scope}
  \begin{scope}[xshift=8cm,yshift=0cm]
    \draw[step=.5cm,lightgray,thick]
    (0,0) grid (4,3);
    \draw[line width=1mm,gray,dashed]
    (0,0) -- (0,3) -- (3.5,3) -- (3.5,0) -- (0,0);
    \draw[line width=1mm,gray,dashed]
    (3.5,0.5) -- (3.5,1) -- (4,1) -- (4,0.5) -- (3.5,0.5);
    \draw[black] (-0.25,1.5) node [left] {$t_1$};
    \draw[black] (4.25,0.75) node [right] {$t_2=1$};
    \draw[black] (1.75,3.25) node [above] {$s_1$};
    \draw[black] (3.75,-0.25) node [below] {$s_2=1$};
    \draw[black] (2,-0.75) node [below] {(\ref{item:CondF})};
    \draw [|-|, black] (4.25,1)  -- (4.25,3)
       node [black,midway,right=4pt] {$m$};
    \draw[black] (1.75,1.5) node {$R'$};
    \draw[black] (3.75,0.75) node {$\tilde{R}''$};
  \end{scope}
  \end{tikzpicture}
  \caption{The alignments of rectangles $R'$ and $\tilde{R}''$ that need to be considered.}
  \label{figure:Cond}
\end{figure}

Before we proceed let us note that by Fact \ref{fact:n_necessary} neither Condition~\ref{item:CondD},~\ref{item:CondE} nor~\ref{item:CondF} could occur if we had $|A| = n$. The first two do not cause us many additional complications and we can deal with them straightforwardly. However, as we shall see, analysing Condition~\ref{item:CondF} is the crucial part of the proof of inequality \eqref{equation:Mnleq}. The possibility of this scenario is what makes the recursive relation for $M(k,\ell)$ much more complicated than the one in \cite{benevidesprzykucki02}. Additionally, it has strong implications when we later try to find an explicit solution to the recursion.

Let us analyse the possible cases one by one. Assume first that Condition~\ref{item:CondB} holds. Note that, in this case,
\[
S(R',\tilde{R}'') =  M(R') + \max\{s_1+t_2-1, s_2+t_1-1\}.
\]

It is easy to check that $S(R',\tilde{R}'')$ cannot decrease if we ``extend'' the rectangle $R'$ and ``shrink'' $\tilde{R}''$. In fact, when $\max\{s_1, t_1\} \ge 2$ then we can use Claim \ref{claim:increaseby1} and so, for any $i < s_2$ and $j < t_2$, we have $M(s_1 + i, t_1+j) \geq M(s_1, t_1) + i+j$. Together with 
\[
 \max\{(s_1+i)+(t_2-j)-1, (s_2-i)+(t_1+j)-1\} \geq \max\{s_1+t_2-1, s_2+t_1-1\} - \max\{i,j\},
\]
we conclude that the largest value of $S(R',\tilde{R}'')$ is given when $\tilde{R}''$ is a single site. Therefore, $S(R', \tilde{R}'') \le M(k-1,\ell-1) + \max\{k,\ell\} - 1$. When $\max\{s_1, t_1\} = 1$ then $R'$ is a single site. Since we assume $M(R') \geq M(\tilde{R}'')$ we would require $\tilde{R}'' \in \Rec(1,1) \cup \Rec(1,2) \cup \Rec(2,1)$. This yields $\max\{k,\ell\} \leq 3$ which contradicts our assumption that $k, \ell \ge 3$ with $(k, \ell) \neq (3,3)$.

Now, assume that Condition~\ref{item:CondA} (or its analogue with $k$ and $\ell$ swapped) holds. Observe that in this case
\begin{equation*}
 S(R', \tilde{R}'') = \begin{cases}
    M(R') + \max\{s_1+t_2, s_2+t_1\},     & \text{if } t_1, t_2 \geq 2, \\
    M(R') + s_2 +t_1,     & \text{if } t_2 = 1, \\
    M(R') + s_1 +t_2,     & \text{if } t_1 = 1.
  \end{cases}
\end{equation*}

If $t_1, t_2 \ge 2$ then it is easy to reduce this case to the previous one: by Claim \ref{claim:increaseby1} we have $M(s_1 + 1, t_1) \geq M(s_1, t_1) + 1$, while
\[
 \max\{(s_1 + 1) + (t_2 - 1) - 1, s_2 + t_1 - 1\} = \max \{s_1 + t_2, s_2 + t_1\} - 1.
\]
Putting these inequalities together we have $S(R', \tilde{R}'') \le S(R^+, R^-)$ where $R^+ \in \Rec(s_1+1,t_1)$ and $R^- \in \Rec(s_2,t_2-1)$, and where $R^+, R^-$ satisfy Condition~\ref{item:CondB}. If $t_2 = 1$, then $t_1 \ge 3$ (recall, $k, \ell \ge 3$). Hence, as in the case of Condition~\ref{item:CondB}, we can use Claim~\ref{claim:increaseby1} and extend $R'$ rightwards to bound $S(R', \tilde{R}'')$ from above using the case where $\tilde{R}''$ is a single site and obtain $S(R', \tilde{R}'') \le M(k-2, \ell) + \ell+1$. Note that swapping $k$ and $\ell$ gives the bound $S(R', \tilde{R}'') \le M(k, \ell-2) + k+1$.

Finally, if $t_1 = 1$ then $t_2 \ge 3$ and, since $M(R') \geq M(\tilde{R}'')$, also $s_2 = 1$. In this case all corners of $R'$ and $\tilde{R}''$ must be initially infected and we can improve the bound $S(R', \tilde{R}'') = M(R') + s_1 + t_2$ to $s_1 +t_2 = k+\ell-2$. Then, $R$ becomes infected after at most $k+\ell-2$ steps which is not more than $M(k-1, \ell-1) + \max\{k,\ell\}-1$ for all $k, \ell \geq 3$.

Suppose now that Condition~\ref{item:CondC} holds. Note that, for a fixed $R'$ and given $m$, we have $S(R', \tilde{R}'') = M(R') + \max\{m+s_2+1, t_1-m+s_2\}$ which is maximum when $m=0$ or $m = t_1-1$ and this case is equivalent to Condition~\ref{item:CondA} with $t_2 = 1$.

Hence we see that
\begin{equation}
\label{eq:boundABC}
  \max \begin{cases}
 M(k-1,\ell-1) + \max\{k,\ell\} - 1,\\
 M(k-2, \ell) + \ell+1,\\
 M(k,\ell-2) + k+1,
 \end{cases} 
\end{equation}
is the maximum percolation time in $[k] \times [\ell]$ when we restrict ourselves to rectangles $R'$ and $\tilde{R}''$ that satisfy Conditions \ref{item:CondB}, \ref{item:CondA} or \ref{item:CondC}.

We shall prove that the same bound applies when either Condition \ref{item:CondD} or \ref{item:CondE} holds, reducing the analysis of those cases to one of Conditions \ref{item:CondB}, \ref{item:CondA} or \ref{item:CondC}. Thus, consider the case when Condition~\ref{item:CondD} applies to $R'$ and $\tilde{R}''$. Recall that $M(R') \geq M(\tilde{R}'')$. Given $m$ we have
\[
 S(R',\tilde{R}'') = M(R') + \max\{s_1+m+1, s_1+t_2-m-t_1+1\}.
\]
that attains its maximum when $m=0$ or $m = t_2 - t_1$. However, for these values of $m$ we could further shrink $\tilde{R}''$ by setting $t_2 = \ell-t_1+1$ and hence reducing this case to the one where Condition~\ref{item:CondA} holds. (If $m=0$ and $t_1 = 1$ then $R_1$ and $\tilde{R}''$ already satisfy both Condition~\ref{item:CondD} and~\ref{item:CondA}.) 

We deal with $R'$ and $R''$ satisfying Condition~\ref{item:CondE} in an analogous way, bounding $S(R',\tilde{R}'')$ from above by taking $m=0$ and then reducing it to the case where Condition~\ref{item:CondB} is satisfied.

Finally let us consider the case where Condition \ref{item:CondF}, or its version with $k$ and $\ell$ swapped, applies to $R'$ and $\tilde{R}''$. In this case we need to be more careful: using similar arguments as before we can only conclude that 

\begin{equation} \label{eq:simpleFbound}
 S(R',\tilde{R}'') =
 \begin{cases}
    M(R') + \max\{m, \ell-m-1\} \le M(R')+\ell-1, & \text{if $R' \in \Rec(k-1, \ell)$}, \\
    M(R') + \max\{m, k-m-1\} \le M(R')+k-1,     & \text{if $R' \in \Rec(k, \ell-1)$}.
  \end{cases}
\end{equation}
Unfortunately this bound is not good enough to prove inequality \eqref{equation:Mnleq}. To improve it we need to analyse how the proximity of $\tilde{R}''$ affects the infection process inside $R'$.

Recall that we initially chose $R'$ and $R''$ together with $A', A'' \subsetneq A$ spanning them according to Proposition~\ref{prop:rectangles}. We later chose $\tilde{R}'' \subset R''$ and we assumed that Condition~\ref{item:CondF} applies to $R'$ and $\tilde{R}''$. However, when $R' = [k-1] \times [\ell]$ then we see that $A''$ must contain a site of the form $(k,i)$ for some $1 \le i \le \ell$ (this site is of course disjoint from $R'$). This is because $R'$ and $R''$ together span~$R$. Thus we in fact can assume that $R'$ (internally spanned by $A'$) and $R''$ (which is a single site) satisfy Condition~\ref{item:CondF} (ignoring the influence of some sites in $A''$ could not decrease the percolation time of $A$).

To continue the analysis of this case we shall need the following claim.

\begin{claim} \label{claim:cornersLast}
Let $A$ be a set of sites percolating in $R = [k] \times [\ell]$ where $k, \ell \geq 2$. Then for any site $(i,j) \in R \setminus \{(1,1),(1,\ell),(k,1),(k,\ell)\}$ we have $I_A(i,j) \leq M(k,\ell)-1$.
\end{claim}
{\par{\it Proof of Claim~\ref{claim:cornersLast}}. \ignorespaces}
It is enough to prove the claim for all percolating sets minimal under containment (as for any $A \subset B$ we have $I_B(i,j) \le I_A(i,j)$ for all $i, j$). Let $A$ be such a set. Applying Proposition \ref{prop:rectangles} to $R$ and $A$ we obtain sets $A'$ and $A''$ that partition $A$ and internally span two rectangles $R', R'' \subsetneq R$ such that $\closure{R' \cup R''} = R$. Note that, by the minimality of $A$, the set $R \setminus (R' \cup R'')$ contains no initially infected sites.

If $k=\ell=2$ then all sites in $[k] \times [\ell]$ are corners and the claim is trivial. If, without loss of generality, $k>2$ then $M(k,\ell)>1$. By Claim \ref{claim:increaseby1} we have $\max\{M(R'),M(R'')\}<M(k,\ell)$. Hence for any $(i,j) \in R' \cup R''$ we have $I_A(i,j)\leq \max\{M(R'),M(R'')\}<M(k,\ell)$.
Now, let
\[
B = R \setminus \left( R' \cup R'' \cup \{(1,1),(1,\ell),(k,1),(k,\ell)\}\right).
\]
If $\{(1,1),(1,\ell),(k,1),(k,\ell)\} \subset R' \cup R''$ and $B \neq \emptyset$ then $\Phi(R'), \Phi(R'') \leq k+\ell-2$ (see Figure \ref{figure:cornersOccupied}) and therefore by Claim \ref{claim:increaseby1} we have $M(R'),M(R'') \leq M(R)-2$ and hence for any $(i,j) \in B$ we have $I_A(i,j) \leq M(k,\ell)-1$. Thus we may assume that $R \setminus (R' \cup R'')$ contains some corner site of $R$. Let $(i,j)$ be any site of $B$. We consider the two following cases:
\begin{itemize}
 \item If $\dist(R',R'')=2$ then $M(R'),M(R'') \leq M(R)-2$. Thus, if we have $\dist((i,j),R') = \dist((i,j),R'') = 1$ then
\[
 I_A(i,j) \leq \max\{M(R'),M(R'')\} + 1 \leq M(k,\ell)-1.
\]
 \item If either $\dist(R',R'')=2$ and $\dist((i,j),R') \neq 1$ or $\dist((i,j),R'') \neq 1$, or if $\dist(R',R'') \neq 2$, then no matter how the rectangles $R'$ and $R''$ are aligned we can find a corner site $(k',\ell') \in R \setminus (R' \cup R'')$ such that to infect $(k',\ell')$ in the process we need to infect $(i,j)$ first. This follows from the fact that the rectangular region in $R \setminus (R' \cup R'')$ that contains $(k',\ell')$ becomes infected starting from its own corner opposite $(k',\ell')$. Thus $I_A(i,j) < I_A(k',\ell') \leq M(k,\ell)$.
\end{itemize}
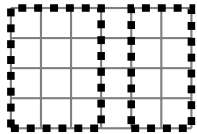
\begin{figure}[ht]
    \centering
    \begin{tikzpicture}[scale=\tikzscale]
      \draw[step=.5cm,gray,thick]
      (0,0) grid (3,2);
      \draw[line width=1mm,black,dashed]
      (0,0) -- (1.5,0) -- (1.5,2) -- (0,2) -- (0,0);
      \draw[line width=1mm,black,dashed]
      (2,0) -- (3,0) -- (3,2) -- (2,2) -- (2,0);
    \end{tikzpicture}
    \caption{The alignment of $R'$ and $R''$ containing all $4$ corner sites}  
    \label{figure:cornersOccupied}
\end{figure}
This completes the proof of the claim.
\endproof

An important consequence of Claim \ref{claim:cornersLast} is that when rectangles $R'$ and $R''$ in $R$ satisfy Condition~\ref{item:CondF} then, no matter how we locate $R''$ in $R$, the infection of $R \setminus (R' \cup R'')$ starts at the latest at time $M(R')-1$. This improves the bound on the time that $A$ takes to percolate given in inequality \eqref{eq:simpleFbound} to
\begin{equation} \label{eq:yetsimpleFbound}
 S(R', R'') \le \max
			 \begin{cases}
	    		M(R')+\ell-2, & \text{if $R' \in \Rec(k-1, \ell)$} \\
	    		M(R')+k-2     & \text{if $R' \in \Rec(k, \ell-1)$} \\
		  	 \end{cases}
\end{equation}

To finish the proof of Theorem \ref{theorem:Mnrecurrence} we apply Proposition \ref{prop:rectangles} to $R'$ (we can do this as $k,\ell \geq 3$ and $R''$ is a single site). So let $A'$ be partitioned into sets $A'_1$ and $A'_2$ spanning rectangles $R'_1$ and $R'_2$ respectively, satisfying Proposition \ref{prop:rectangles}. Assume that $M(R'_1) \geq M(R'_2)$.

If $R'_1$ and $R'_2$ satisfy Condition~\ref{item:CondF} inside $R'$, with $R'_2$ being a single site, then we can bound the (total) time that $A$ takes to percolate in a much better way than using inequality \eqref{eq:yetsimpleFbound}. In fact, considering the possible cases it can, again, be bounded from above by \eqref{eq:boundABC}. This follows from the fact that $\dist(R_1', R'') \le 2$ and therefore, with $R_1'$ fully infected, the processes of infecting $R' \setminus (R'_1 \cup R'_2)$ and $\closure{R'_1 \cup R''} \setminus (R'_1 \cup R'')$ run simultaneously.

In the remainder we assume that $R_1'$ and $R_2'$ satisfy one of the conditions (\ref{item:CondB})-(\ref{item:CondE}) in $R'$ and we improve the bound \eqref{eq:yetsimpleFbound} by replacing $M(R')$ with a better bound on the time that $A'$ takes to percolate in $R'$.

If $R'_1$ and $R'_2$ satisfy Condition~\ref{item:CondA} or~\ref{item:CondC} in $R'$ then, by what we already know about the bounds for these conditions (i.e., that the upper bound on $M(R')$ is the weakest when $R'_2$ is a single site, see \eqref{eq:boundABC} with the dimensions of $R'$ in place of $k$ and $\ell$), the bound in \eqref{eq:yetsimpleFbound} is at most
\begin{equation*}
\max
  \begin{cases}
    M(k-2,\ell-1) + k+\ell-2,\\
    M(k-1,\ell-2) + k+\ell-2,\\
    M(k,\ell-3) + 2k-1,\\
    M(k-3,\ell) + 2\ell-1.
  \end{cases}
\end{equation*}

If $R'_1 \in \Rec(s_1,t_1)$ and $R'_2 \in \Rec(s_2,t_2)$ inside $R'$ satisfy Condition~\ref{item:CondB} then $R''$, $R'_1$ and $R'_2$ are (up to simple rotations), for some $m \leq t_1+t_2-1$, mutually aligned as in Figure \ref{figure:CondBF} (where $R''$ is depicted with a shaded square).
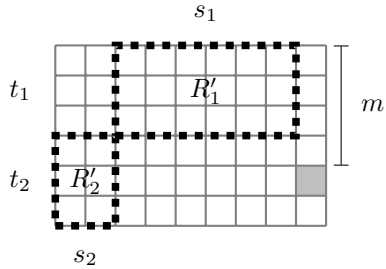
\begin{figure}[ht]
  \centering
    \begin{tikzpicture}[scale=\tikzscale]
      \fill[color=lightgray]  (4,0.5) rectangle +(0.5,0.5);
      \draw[step=.5cm,gray,thick]
      (0,0) grid (4.5,3);
      \draw[line width=1mm,black,dashed]
      (0,0) -- (1,0) -- (1,1.5) -- (0,1.5) -- (0,0);
      \draw[line width=1mm,black,dashed]
      (1,1.5) -- (4,1.5) -- (4,3) -- (1,3) -- (1,1.5);
      \draw[black] (-0.25,2.25) node [left] {$t_1$};
      \draw[black] (-0.25,0.75) node [left] {$t_2$};
      \draw[black] (2.5,3.25) node [above] {$s_1$};
      \draw[black] (0.5,-0.25) node [below] {$s_2$};
      \draw [|-|, black] (4.75,1)  -- (4.75,3)
         node [black,midway,right=4pt] {$m$};
      \draw[black] (2.5,2.25) node {$R'_1$};
      \draw[black] (0.5,0.75) node {$R'_2$};
    \end{tikzpicture}
    \caption{Condition~\ref{item:CondB} followed by Condition~\ref{item:CondF}}
    \label{figure:CondBF}
\end{figure}

Let us analyse the possible cases conditioning on the value of $t_2$. If $t_2 = 1$ then we have $\dist(R'_1, R'') \leq 2$ so the infection of $\closure{R'_1 \cup R''} \setminus (R'_1 \cup R'')$ starts at the latest at time $M(R'_1)$. As before, by Claim \ref{claim:increaseby1}, we can shrink $R''_2$, and so the upper bound on percolation time is maximized for $s_2=1$. In this case the largest bound on $M(k,\ell)$ is achieved when $m=0$ and is equal to
\begin{equation*}
\max
  \begin{cases}
    M(k-1,\ell-2) + \max\{k,\ell\}-1\\
    M(k-2,\ell-1) + \max\{k,\ell\}-1
  \end{cases}
 < M(k-1,\ell-1) + \max\{k,\ell\}-1.
\end{equation*}

If $t_2 > 1$ then, by Claim \ref{claim:increaseby1} and Claim \ref{claim:cornersLast}, the bound on percolation time is maximized either for $t_2=2$, $s_2=1$ and $m=t_1+t_2-1$ which as the upper bound on $M(k,\ell)$ gives
\begin{equation*}
    M(k-2,\ell-2) + k + \ell-3 \leq M(k-2,\ell-1) + k + \ell-2,
\end{equation*}
or for $s_1=2$, $t_1=1$ and $m=t_1+t_2-1$ which as the upper bound gives
\begin{equation*}
\max
  \begin{cases}
    M(k-1,\ell-3) + 2k - 2 \\ 
    M(k-3,\ell-1) + 2 \ell -2
  \end{cases}
 \leq \max
  \begin{cases}
    M(k,\ell-3) + 2k - 1 \\ 
    M(k-3,\ell) + 2 \ell -1
  \end{cases},
\end{equation*}
or for $s_1=1$, $t_1=1$ and $m=t_1+t_2-1$ which as the upper bound gives
\begin{equation*}
\max
  \begin{cases}
    M(k-1,\ell-2) + \max\{k+1,\ell-2\} \\ 
    M(k-2,\ell-1) + \max\{k-2,\ell+1\}
  \end{cases} \leq \max
  \begin{cases}
    M(k-1,\ell-2) + k+\ell-2 \\ 
    M(k-2,\ell-1) + k+\ell-2
  \end{cases}.
\end{equation*}
Thus the upper bound on the percolation time of $A$ obtained when Condition~\ref{item:CondB} holds for $R_1', R_2'$ inside $R'$ is at most $W(k,\ell)$, i.e., the maximum in the right-hand-side of equation \eqref{equation:Mnrecurrence}.

Finally, if $R'_1$ and $R'_2$ inside $R'$ satisfy Condition~\ref{item:CondD} or~\ref{item:CondE} with $M(R'_1) \geq M(R'_2)$ then, as already noted, by setting $m=0$ and shrinking $R'_2$ we can bound from above the percolation time of $A'$ by the bounds obtained under conditions~\ref{item:CondB} and~\ref{item:CondA}. That completes the proof of the upper bound on $M(k, \ell)$ and of Theorem~\ref{theorem:Mnrecurrence}.
\end{proof}

\begin{remark}
 Relation \eqref{equation:Mnrecurrence} does not allow us to immediately give an exact formula for $M(n)$. However, with the use of a computer we can evaluate $M(n)$ and at the same time find an $(n,n)$-perfect set. Our simulations suggest that these sets have size approximately $\frac{23n}{18}+O(1)$ (for example, for $n=1000$ it is $1277$). In the next section we find the asymptotic formula for $M(n)$. For the lower bound we shall use sets similar to those suggested by our computations.
\end{remark}

\section{Computing the asymptotic value of $M(n)$}
\label{sec:valueMnn}
In this section we use the existence of $(n,n)$-perfect sets to compute the asymptotic value of $M(n)$. We say that a $(k,\ell)$-perfect set $A$ together with the sequence of rectangles $P_0 \subset P_1 \subset  \ldots \subset  P_r \in \Rec(k,\ell)$ associated with it are \textit{described by a triple} $(s_0, t_0, m_1 m_2 \ldots m_r)$ if $P_0 \in \Rec(s_0, t_0)$ and additionally, for $1\le i \le r$, Move~$m_i$ is used to obtain $P_{i}$ from $P_{i-1}$. We write $T_0 = M(P_0)$ and, for $i \ge 1$, we denote by $T_i$ the additional time it takes to infect the sites of $P_i$ after all sites of $P_{i-1}$ are infected. We say that $T_0, T_1,\ldots, T_r$ is the \emph{time sequence} of $A$. Finally, we say that a triple $(s_0, t_0, m_1 m_2 \ldots m_r)$ is a \emph{scheme} solving $M(k,\ell)$ if it describes a $(k,\ell)$-perfect set.

Note that a triple $(s_0, t_0, m_1 m_2 \ldots m_r)$ may describe multiple $(n,n)$-perfect sets since it only determines the dimensions of the rectangles $P_i$ but not their precise coordinates. Nevertheless, all $(n,n)$-perfect sets described by $(s_0, t_0, m_1 m_2 \ldots m_r)$ have the same time sequence. Note that if $T_0, T_1,\ldots, T_r$ is a time sequence of an $(n,n)$-perfect set then $M(n) = \sum_{i=0}^r T_i$.

\begin{observation}\label{obs:minimalityprinciple} Let $(s_0, t_0, m_1 m_2 \ldots m_r)$ be a scheme and $P_0 \subset P_1 \subset  \ldots \subset  P_r$ be the sequence of rectangles generated by it. Then for any $1 \le j \le r$ the triple $(s_0, t_0, m_1 m_2 \ldots m_j)$ is a scheme. In particular, it describes a set that percolates $P_j$ in maximum time.
\end{observation}

\begin{remark}\label{rem:smallMnn} In Appendix~\ref{appendix} we consider a number of small cases and show that for any $k, \ell \geq 3$, $(k,\ell) \neq (3,3)$, there exists a scheme $(s_0, t_0, m_1 m_2 \ldots m_r)$ that solves $M(k,\ell)$ and is such that either $s_0 \ge 3$ and $t_0 = 2$ or $s_0 = 2$ and $t_0 \geq 3$.
\end{remark}

Let $a$, $b$ be natural numbers and let $x_1 \ldots x_a$ and  $y_1\ldots y_b$ be sequences of moves. We say that these sequences are compatible if applying moves $x_1 \ldots x_a$ to a certain rectangle $R$ yields a rectangle with the same dimensions as when applying moves $y_1\ldots y_b$ to $R$. For example, for any $1\le i, j \le 7$, the sequence $ij$ is compatible with $ji$, the sequence $61$ is compatible with $35$, the sequence $111$ is compatible with $45$, but $12$ is not compatible with $13$ (because the order of dimensions of the resulting rectangle matters).

Fix $1 \le i \le r$ and let $P_i \in \Rec(k, \ell)$. Clearly the value of $T_i$ depends only on $k$, $\ell$ and $m_i$. We list its possible values in Table~\ref{table1} (see also equation \eqref{equation:Mnrecurrence}). For $2 \le i \le r$, applying this argument twice, we can compute the value of $T_{i} + T_{i-1}$ as a function of $k$, $\ell$, $m_i$ and $m_{i-1}$ only. In Table~\ref{table2} we list the values of $T_{i} + T_{i-1}$ for $m_i, m_{i-1} \in \{2, 3, 4, 5, 6, 7\}$ and in Table~\ref{table3} we list the values of $T_{i} + T_{i-1}$ when either $m_i = 1$ or $m_{i-1} = 1$.

\begin{figure}[htb]
\centering
\begin{floatrow}
\capbtabbox{%
\begin{tabular}{|l|c|c|}
    \hline
    $m_i$ & $P_{i-1}$ & $T_i$ \\
    \hline
      & & \\[-1em]
    \hline
    1 & $(k-1,\ell-1)$ & $\max\{k,\ell\}-1$ \\
    2 & $(k-2,\ell)$   & $\ell+1$  \\
    3 & $(k,\ell-2)$   & $k+1$ \\
    4 & $(k-2,\ell-1)$ & $k+\ell-2$ \\
    5 & $(k-1,\ell-2)$ & $k+\ell-2$ \\
    6 & $(k,\ell-3)$   & $2k-1$ \\
    7 & $(k-3,\ell)$   & $2\ell-1$ \\
    \hline
  \end{tabular}
}{%
  \caption{Dimensions of $P_{i-1}$ and value of $T_{i}$ given~$m_i$, assuming that $P_i \in \Rec(k, \ell)$.}
  \label{table1} 
}
\ffigbox{%
  \begin{tikzpicture}[scale=\tikzscale]
  
  \draw[->] (0,-0.5) -- (0,3.5) node[above left] {$\ell$};
  \draw[->] (-0.5,0) -- (3.5,0) node[below right] {$k$};

  \draw[step=1cm,gray,dotted] (0,0) grid (3,3);

  \draw[->, line width=0.5mm] (0,0) -- (1,1) node[above right] {$1$};
  \draw[->, line width=0.5mm] (0,0) -- (2,1) node[right] {$4$};    
  \draw[->, line width=0.5mm] (0,0) -- (1,2) node[above] {$5$};
  \draw[->, line width=0.5mm] (0,0) -- (2,0) node[below left] {$2$};
  \draw[->, line width=0.5mm] (0,0) -- (3,0) node[below left] {$7$};
  \draw[->, line width=0.5mm] (0,0) -- (0,2) node[below left] {$3$};
  \draw[->, line width=0.5mm] (0,0) -- (0,3) node[below left] {$6$};
  \end{tikzpicture}
}{%
  \caption{Direction of each Move.}
  \label{figure:movesdirections}
}
\end{floatrow}
\end{figure}

\begin{table}[ht]\centering
  \begin{tabular}{|c||c|c|c|} 
    \hline
       & $m_i = 2$ & $m_i = 3$ & $m_i = 4$ \\
    \hline
       & & & \\[-1em]
    \hline
    $m_{i-1} = 2$ & $2\ell+2$  & $k+\ell$ & $k+2\ell-2$ \\
    \hline
    $m_{i-1} = 3$ & $k+\ell$ & $2k+2$ & $2k+\ell-3$ \\
    \hline
    $m_{i-1} = 4$ & $k+2\ell-3$ & $2k+\ell-3$ & $2k+2\ell-7$ \\
    \hline
    $m_{i-1} = 5$ & $k+2\ell-3$ & $2k+\ell-3$ & $2k+2\ell-7$ \\
    \hline
    $m_{i-1} = 6$ & $2k+\ell-4$ & $3k$ & $3k+\ell-7$ \\
    \hline
    $m_{i-1} = 7$ & $3\ell$ & $k+2\ell-4$ & $k+3\ell-5$ \\
    \hline
  \end{tabular}
  \vskip 0.1in
  \begin{tabular}{|c||c|c|c|c|c|c|} 
    \hline
       & $m_i = 5$ & $m_i = 6$ & $m_i = 7$ \\
    \hline
       & & & \\[-1em]
    \hline
    $m_{i-1} = 2$ & $k+2\ell-3$ & $2k+\ell-3$ & $3\ell$ \\
    \hline
    $m_{i-1} = 3$ & $2k+\ell-2$ & $3k$ & $k+2\ell-3$ \\
    \hline
    $m_{i-1} = 4$ & $2k+2\ell-7$ & $3k+\ell-6$ & $k+3\ell-6$ \\
    \hline
    $m_{i-1} = 5$ & $2k+2\ell-7$ & $3k+\ell-6$ & $k+3\ell-6$ \\
    \hline
    $m_{i-1} = 6$ & $3k+\ell-5$ & $4k-2$ & $2k+2\ell-8$ \\
    \hline
    $m_{i-1} = 7$ & $k+3\ell-7$ & $2k+2\ell-8$ & $4\ell-2$ \\
    \hline
  \end{tabular}
  \caption{Values of $(T_{i} + T_{i-1})$ for $m_i, m_{i-1} \in \{2, 3, 4, 5, 6, 7\}$, assuming that $P_i \in \Rec(k, \ell)$.}  
  \label{table2} 
\end{table}

\begin{table}[ht]\centering
  \begin{tabular}{|c|c|c|}
    \hline 
      & $(m_{i-1}, m_i) = (j, 1)$ & $(m_{i-1}, m_i) = (1, j)$ \\
    \hline
      & & \\[-1em]
    \hline
    $j = 1$ & $2\max\{k,\ell\}-3$ & $2\max\{k,\ell\}-3$ \\
    $j = 2$ & $\max\{k,\ell\}+\ell-1$   & $\ell+\max\{k,\ell-2\}$ \\
    $j = 3$ & $\max\{k,\ell\}+k-1$   & $k+\max\{k-2,\ell\}$ \\
    $j = 4$ & $\max\{k,\ell\}+k+\ell-5$ & $k+\ell+\max\{k-2,\ell-1\}-3$ \\
    $j = 5$ & $\max\{k,\ell\}+k+\ell-5$ & $k+\ell+\max\{k-1,\ell-2\}-3$ \\
    $j = 6$ & $\max\{k,\ell\}+2k-4$  & $2k+\max\{k,\ell-3\}-2$ \\
    $j = 7$ & $\max\{k,\ell\}+2\ell-4$  & $2\ell+\max\{k-3,\ell\}-2$ \\
    \hline
  \end{tabular}
  \caption{Possible values of $(T_{i} + T_{i-1})$ for $m_i = 1$ or $m_{i-1} = 1$, assuming that $P_i \in \Rec(k, \ell)$.}  
  \label{table3}
\end{table}

Initially the object of our interest in Table~\ref{table2} and~Table~\ref{table3} is whether, for each pair $(a,b)$ with $1 \le a, b \le 7$, for $P_i \in \Rec(k, \ell)$ the value of $(T_{i} + T_{i-1})$ is larger when $(m_{i-1}, m_i) = (a,b)$ or when $(m_{i-1}, m_i) = (b,a)$. We summarize the answer to that question in Figure~\ref{figure:moves} that tells us what pairs of consecutive moves are prohibited in a scheme (because one could swap them and obtain a slower percolating process). A solid directed edge from $a$ to $b$ means that, no matter what the values of $k$ and $\ell$ are, it takes strictly longer to apply Move $b$ right before Move $a$ than it takes to apply them in the opposite order. Thus in this case the consecutive pair of moves $ab$ inside a scheme is prohibited. A dashed directed edge from $a$ to $b$ means that, no matter what the values of $k$ and $\ell$ are, it always takes at least as much time to apply Move $b$ followed by Move $a$ as it takes to do it in the opposite order. Hence $ab$ is 
not prohibited but might be avoided in a scheme. A dashed undirected edge means that the order of moves $a$ and $b$ maximizing the value of $(T_{i} + T_{i-1})$ depends on the values of $k$ and $\ell$. No edge between $a$ and $b$ means that the order we use does not affect the value of $(T_{i} + T_{i-1})$.

\begin{figure}[ht] \centering
  \begin{tikzpicture}[ shorten >=1pt,->]
    \tikzstyle{vertex}=[draw,shape=circle,minimum size=14pt,inner sep=0pt]
    \foreach \name/\x/\y in {1/0/0, 2/-1.12/0, 3/1.1/0, 5/1.5/1.2, 6/-1.5/1.2, 4/-1.5/-1.2, 7/1.5/-1.2}
      \node[vertex] (P-\name) at (\x,\y) {$\name$};
    \draw (P-4) .. controls +(120+180:1.2cm) and +(60+180:1.2cm) .. (P-7);
    \draw (P-7) .. controls +(150+270:1cm) and +(30+270:1cm) .. (P-5);
    \draw (P-5) .. controls +(120:1.2cm) and +(60:1.2cm) .. (P-6);
    \draw (P-6) .. controls +(150+90:1cm) and +(30+90:1cm) .. (P-4);
    \foreach \from/\to in {5/3,4/2, 7/3, 6/2}
      {\draw (P-\from) -- (P-\to);}
    \foreach \from/\to in {4/1,5/1}
      {\draw[dashed] (P-\from) -- (P-\to);}
    \foreach \from/\to in {2/1,1/3,1/6,1/7}
      {\draw[dashed,-] (P-\from) -- (P-\to);}
  \end{tikzpicture}
  \caption{Relation between pairs of consecutive moves $(m_{i-1}, m_i)$ and the value of $(T_{i} + T_{i-1})$.}
  \label{figure:moves}
\end{figure}
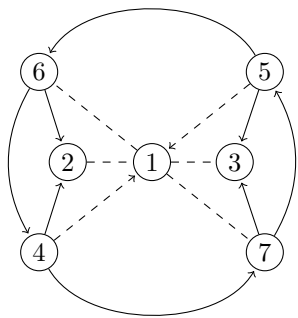

Our proof will deal with sequences of moves and in order to describe these we shall use the following notation. We say that a finite (possibly empty) sequence of moves is of the form $[a_1|a_2|\ldots|a_r]^*$ if all its terms belong to $\{a_1,a_2,\ldots,a_r\} \subseteq [7]$; we say that it is of the form $[a_1|a_2|\ldots|a_r]^{\leq j}$ if, in addition, it has at most $j$ terms. We shall concatenate these expressions to create more general ones that describe the corresponding sets of concatenated sequences of moves. For example, all of the sequences $1444336366$, $43333$, $16633$ are of the form $[1]^{\leq 1}[4]^*[3|6]^*$, but $144334$ is not.

Next, we prove a series of propositions about schemes for $M(k,\ell)$. These propositions will allow us to gain control over the structure of the schemes and consequently, implementing additional machinery, to give tight bounds on $M(n)$.

\begin{prop}\label{prop:step1} For any $k, \ell \geq 3$, $(k,\ell) \neq (3,3)$, there exists a scheme solving $M(k,\ell)$ of the form $(s_0, t_0, [1|2|3]^*[4|5|6|7]^*)$ with $s_0 \geq 3, t_0 = 2$ or $s_0 = 2, t_0 \geq 3$.
\end{prop}
\begin{proof}
Given $k, \ell$, consider a scheme $Q = (s_0,t_0, m_1m_2 \ldots m_r)$ with $s_0 \geq 3, t_0 = 2$ or $s_0 = 2, t_0 \geq 3$ that solves $M(k,\ell)$ (which exists by Remark \ref{rem:smallMnn}) that minimizes the sum $S = \sum_{m_i \in\{1, 2, 3\}} i$. Proposition~\ref{prop:step1} follows immediately from the following claim: in such a scheme, for any $i$ with $2\le i \le s$, if $m_i$ is equal to $1$, $2$ or $3$ then $m_{i-1}$ is equal to $1$, $2$, or $3$. Let us prove that this claim holds.

Fix $2\le i \le r$. Assume first that $m_i = 2$. In Figure~\ref{figure:moves} we see that $m_{i-1} \notin \{4, 6\}$ and that if $m_{i-1} \in \{5,7\}$ then we could swap the order of  $(m_{i-1},m_i)$ without changing percolation time and at the same time decreasing the value of $S$, contradicting the choice of~$Q$. Therefore $m_{i-1}$ must be either $1$, $2$ or $3$. The case where $m_i = 3$ is analogous.

Assume now that $m_i = 1$.  If $m_{i-1} \in \{4, 5\}$ then we could swap the order of  $(m_{i-1},m_i)$ without decreasing percolation time and at the same time decreasing the value of $S$, contradicting the choice of $Q$. Now, suppose that $m_{i-1} = 6$. If $k \ge \ell$ then, as shown in Table~\ref{table3},
\[
 T_{i-1} + T_{i} = \max\{k,\ell\}+2k-4 < 2k+\max\{k,\ell-3\}-2
\]
in which case we could set $(m_{i-1},m_i) = (1,6)$ and increase percolation time. If $k < \ell$ then again in Table \ref{table3} we find that
\[
  T_{i-1} + T_{i} = \max\{k,\ell\}+2k-4 < 2k+\ell-2 
\]
in which case we can set $(m_{i-1},m_i) = (3,5)$ and increase percolation time. In either case, we contradict the fact that $Q$ is a scheme. Therefore $m_{i-1} \ne 6$. We show that $m_{i-1} \ne 7$ in an analogous way: one could either swap $(7,1)$ or replace it by $(2,4)$ in order to increase percolation time (doing one or the other depending on the values of $k$ and $\ell$). Therefore we must have $m_{i-1}$ equal to $1$, $2$ or $3$.
\end{proof}

Before we continue our investigations of the form of the schemes that solve $M(k,\ell)$ let us make the following two observations about the infection process started from a $(k,\ell)$-perfect set.

\begin{observation}
\label{observation:maximumtwo}
For any $i \ge 1$, no matter which move $(1-7)$ is used at moment $i$, between time step $M(P_{i-1})+1$ and time step $M(P_{i})$ (when the infection of the rectangle $P_{i}$ is complete) at each step at most two new sites become infected.
\end{observation}
\begin{observation} \label{observation:howmanyones}
For any $i \ge 1$, if $s_{i-1}, t_{i-1} \geq 2$, then the following statements hold.
\begin{enumerate}
\item If we use Move $1$ at moment $i$ then there are exactly $|s_{i-1} - t_{i-1}|$ time steps between $M(P_{i-1})+1$ and $M(P_{i})$ (when all sites of $P_i$ are infected) when only one new site becomes infected. These are $M(P_{i})-|s_{i-1} - t_{i-1}|+1, M(P_{i})-|s_{i-1} - t_{i-1}|+2, \ldots, M(P_{i})$.
\item If we use Move $2$ or $3$ at moment $i$ then there are exactly $3$ time steps between $M(P_{i-1})+1$ and $M(P_{i})$ (when all sites of $P_i$ are infected) when only one new site becomes infected. These are $M(P_{i-1})+1, M(P_{i-1})+2, M(P_{i})$.
\end{enumerate}
\end{observation}

From Observation \ref{observation:maximumtwo} and Observation \ref{observation:howmanyones} the following claim follows. Its proof is simple but rather technical and fully analogous to Claim 13 in \cite{benevidesprzykucki02} therefore, for the sake of brevity, we leave it without proof.

\begin{claim}\label{claim:moves123}
Suppose that there exists a $(k,\ell)$-perfect set $A$ internally spanning a rectangle $R \in \Rec(k,\ell)$ with a sequence of rectangles $P_0 \subset P_1 \subset \ldots \subset P_r \in \Rec(k,\ell)$ associated with it, described by a triple of the form $(s_0, t_0, [1|2|3]^*)$ with $s_0 \geq 3, t_0=2$ or $s_0 = 2, t_0 \geq 3$. Then there exists a $(k,\ell)$-perfect set $A'$ internally spanning the rectangle $R \in \Rec(k,\ell)$ described by a triple of the form $(s_0, t_0, [2]^*[1]^*[3]^*)$, or of the form $(s_0, t_0, [3]^*[1]^*[2]^*)$.
\end{claim}

\begin{prop}\label{prop:step2} For any $n\ge 4$ there exists a scheme $Q$ either of the form $(s_0, 2, [1]^*[3]^*[4|5|6|7]^*)$ or of the form $(s_0, 2, [3]^*[1]^*[2]^*[4|5|6|7]^*)$ with $s_0 \ge 3$ that solves $M(n)$.
\end{prop}
\begin{proof}
Consider a scheme $Q = (s_0, 2, m_1m_2 \ldots m_r)$ with $s_0 \geq 3$ and sequence $m_1m_2 \ldots m_r$ of the form $[1|2|3]^*[4|5|6|7]^*$ which exists by Proposition~\ref{prop:step1} (by symmetry, when $k=\ell=n$, we might assume $t_0 = 2$).

Let $j = \max \{i: m_i \in \{1,2,3\}\}$. By Observation~\ref{obs:minimalityprinciple} the sequence of moves $m_1 \ldots m_j$ is such that the time taken to infect $P_j$ is maximum. Therefore, by Claim~\ref{claim:moves123}, we see that we may take $m_1 \ldots m_j$ of the form $[2]^*[1]^*[3]^*$ or of the form $[3]^*[1]^*[2]^*$. We observe that in the first case we obtain a scheme $Q'$ of the form $(s'_0, 2, [1]^*[3]^*[4|5|6|7]^*)$, as the triple $(s_0, 2, [2]^*[1]^*[3]^*[4|5|6|7]^*)$ gets simplified to $(s'_0, 2, [1]^*[3]^*[4|5|6|7]^*)$ (where $s'_0 = s_0+2a$ for $a$ equal to the number of times that Move~2 occurs in $m_1 \ldots m_j$). In the second case we have a scheme of the form $(s'_0, 2, [3]^*[1]^*[2]^*[4|5|6|7]^*)$.
\end{proof}

\begin{prop}\label{prop:step3} For any $n\ge 4$ there exists a scheme $Q$ solving $M(n)$ that is either of the form $(s_0, 2, [1]^{\leq 1}[3]^{\leq 2}[4|5|6|7]^*)$ or of the form $(s_0, 2, [3]^{\leq 2}[1]^{\leq 1}[2]^*[4|5|6|7]^*)$ with $s_0 \ge 3$.
\end{prop}
\begin{proof}
By Proposition \ref{prop:step2} there exists a scheme $Q = (s_0, t_0, m_1m_2 \ldots m_r)$ that is either of the form $(s_0, 2, [1]^*[3]^*[4|5|6|7]^*)$ or of the form $(s_0, 2, [3]^*[1]^*[2]^*[4|5|6|7]^*)$. Let us consider these cases separately.

Assume first that there exists $Q$ of the form $(s_0, 2, [1]^*[3]^*[4|5|6|7]^*)$, and choose one for which the number of times it uses Move~$1$ is minimal. Let $j = \max \{i: m_i=1\}$. Let $P_j \in \Rec(s_j,t_j)$. Assume that Move~$3$ was used at least three times. For $s_j \geq 5$, we could replace the last occurrence of the sequence $333$ by the compatible sequence $66$ without decreasing percolation time. For $3 \leq s_j \leq 4$, we consider all possible options for $Q' = (s_0, t_0, m_1\ldots m_j)$, and note that either:
\begin{enumerate}
 \item $Q'=(3,2,333)$ which takes strictly less time ($15$ steps) to span $R \in \Rec(3, 8)$ than $Q''=(2,7,1)$ does ($16$ steps), or
 \item $Q'=(3,2,1333)$ which takes strictly less time ($21$ steps) to span $R \in \Rec(4, 9)$ than $Q''=(2,9,2)$ does ($22$ steps), or
 \item $Q'=(4,2,333)$ which takes strictly less time ($19$ steps) to span $R \in \Rec(4, 8)$ than $Q''=(2,5,15)$ does ($21$ steps).
\end{enumerate}

By Observation~\ref{obs:minimalityprinciple} none of the above $Q'$ can be an initial segment of $Q$. Thus there must exist $Q$ of the form $(s_0, 2, [1]^*[3]^{\leq 2}[4|5|6|7]^*)$. Now, assume that Move~$1$ is used at least twice, say, $Q$ is of the form $(s_0, 2, 11m_3m_4\ldots m_r)$. If $s_0 \geq 4$, then $Q$ can be replaced by $(s_0-1, 2, 14m_3m_4\ldots m_r)$ for which we still have $P_2 \in \Rec(s_0+2,4)$ and the percolation time of which is at least as big as for $Q$ because
\[
T_0 + T_1 + T_2 = M(s_0,2)+s_0+(s_0+1) = \floor{\frac{7 s_0-1}{2}}
\]
and the time sequence of the modified sequence of moves gives
\[
 T'_0 + T'_1 + T'_2 = M(s_0-1,2)+(s_0-1)+((s_0+2)+4-2) = \floor{\frac{7 s_0}{2}}.
\]
In fact, since in Figure \ref{figure:moves} there is a dashed directed edge from $4$ to $1$ and no edge between $4$ and $3$ we can move the new Move $4$ further in the sequence and obtain $\tilde{Q}$ of the form $(s_0, 2, [1]^*[3]^{\leq 2}[4|5|6|7]^*)$ in which the number of times that we use Move 1 is strictly smaller than in $Q$. This contradicts the minimality of the number of Move~$1$s used in $Q$. Finally, if $s_0 = 3$ then it is enough to notice that $(3,2,11)$ takes strictly less time ($10$ steps) to percolate in $R \in \Rec(5,4)$ than $(5,2,3)$ does ($12$ steps). Therefore Move~1 must be used at most once. Thus $Q$ is of the form $(s_0, 2, [1]^{\leq 1}[3]^{\leq 2}[4|5|6|7]^*)$ as stated.

Hence let us assume that there exists a scheme $Q$ of the form $(s_0, 2, [3]^*[1]^*[2]^*[4|5|6|7]^*)$. By the same argument as in the first case we can conclude that Move~3 is used at most two times. In fact, the only difference is that here we do not need to consider the subcase $Q' = (3,2,1333)$ in our analysis. Therefore there must exist a scheme of the form $(s_0, 2, [3]^{\leq 2}[1]^*[2]^*[4|5|6|7]^*)$.

Assume that Move~$1$ is used at least twice. If Move~3 is not used then $Q$ is of the form $(s_0, 2, 11m_3m_4\ldots m_r)$ and we can get a contradiction as in the first case. So, Move~3 must be used once or twice. It follows from Observation~\ref{observation:howmanyones} that, when we limit ourselves to sequences of the form $(s_0,2,[1|3]^*)$, the slowest sequences are obtained when Move~$1$s are applied to rectangles in which the difference between the length of their longer and their shorter side is maximum. This means that Move~$3$s could be used before Move 1 only if after using them the difference between the lengths of the sides of the rectangle we obtained was at least as large as $s_0-t_0 = s_0-2$. However, since Move~3 is used at most twice then, unless $s_0$ is small, by putting Move~$1$s before $3$s we obtain a sequence slower than if we did it the other way. More precisely, the only cases in which putting Move~$3$s before $1$s could possibly increase the percolation time are those where $s_0-2 < 
3$ and the initial sequences of steps in $Q$ are:
\begin{enumerate}
\item $Q'=(3,2,311)$ which takes strictly less time ($16$ steps) to span $R \in \Rec(5, 6)$ than $Q''=(2,5,12)$ does ($18$ steps), or
\item $Q'=(3,2,3311)$ which takes strictly less time ($24$ steps) to span $R \in \Rec(5, 8)$ than $Q''=(2,3,155)$ does ($25$ steps), or
\item $Q'=(4,2,3311)$ which takes strictly less time ($27$ steps) to span $R \in \Rec(6, 8)$ than $Q''=(2,7,17)$ does ($31$ steps).
\end{enumerate}
As in the first case, sets described by triples $Q''$ span the same rectangles as those spanned by sets described by corresponding triples $Q'$. Thus we see that the triples $Q'$ are not initial segments of schemes. This implies that Move~1 is used at most once, that is, in the second case $Q$ is of the form $(s_0, 2, [3]^{\leq 2}[1]^{\leq 1}[2]^*[4|5|6|7]^*)$ as stated.
\end{proof}

We are now ready to prove our main result.
{\par{\it Proof of Theorem \ref{theorem:arbitraryAsymptotics}}. \ignorespaces}
We begin proving that $M(n) \geq \frac{13}{18}n^2 + O(n)$ by constructing a particular family of percolating sets described by triples of the form $(s_0, 2, 1[4]^*[6]^*)$. (However, these sets are not necessarily $(n,n)$-perfect.) We consider the following way of spanning $[n]^2$ for $n \geq 6$:
\begin{enumerate}
\item choose a natural number $s \in (\frac{n}{3}-3,\frac{n}{3}+3]$ such that $6 | n+s-5$ (note that, in particular, this implies $2 | n-s-1$),
\item in Phase 1 span a rectangle $P_0 \in \Rec(s,2)$ in the maximum possible time,
\item in Phase 2 obtain $P_1 \in \Rec(s+1,3)$ by applying Move $1$ to $P_0$,
\item in Phase 3 obtain $P_{\frac{n-s+1}{2}} \in \Rec(n,\frac{n-s+5}{2})$  by applying Move $4$ $\frac{n-s-1}{2}$ times,
\item in Phase 4 obtain $P_{\frac{2n-s-1}{3}} = [n]^2$ by applying Move $6$ $\frac{n+s-5}{6}$ times.
\end{enumerate}
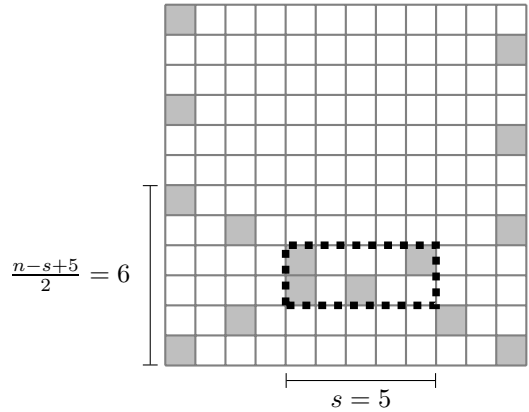
\begin{figure}[ht]
  \centering
  \begin{tikzpicture}[scale=\tikzscale]
    \fill[color=lightgray] (2,1) rectangle (2.5,1.5);
    \fill[color=lightgray] (2,1.5) rectangle (2.5,2);
    \fill[color=lightgray] (3,1) rectangle (3.5,1.5);
    \fill[color=lightgray] (4,1.5) rectangle (4.5,2);
    \fill[color=lightgray] (4.5,0.5) rectangle (5,1);
    
    \fill[color=lightgray] (1,2) rectangle (1.5,2.5);
    \fill[color=lightgray] (1,0.5) rectangle (1.5,1);
    \fill[color=lightgray] (5.5,0) rectangle (6,0.5);
    \fill[color=lightgray] (5.5,2) rectangle (6,2.5);
    \fill[color=lightgray] (0,2.5) rectangle (0.5,3);
    \fill[color=lightgray] (0,0) rectangle (0.5,0.5);
    
    \fill[color=lightgray] (0,4) rectangle (0.5,4.5);
    \fill[color=lightgray] (5.5,3.5) rectangle (6,4);
    \fill[color=lightgray] (0,5.5) rectangle (0.5,6);
    \fill[color=lightgray] (5.5,5) rectangle (6,5.5);
    
    \draw[step=.5cm,gray,thick]
    (0,0) grid (6,6);
    
    \draw[line width=1mm,black,dashed]
    (2,1) -- (2,2) -- (4.5,2) -- (4.5,1) -- (2,1);
      
    \draw [|-|, black]
      (2,-0.25)  -- (4.5,-0.25)
       node [black,midway,below] {$s=5$};
     \draw [|-|, black]
      (-0.25,0)  -- (-0.25,3)
       node [black,midway,left=4pt] {$\frac{n-s+5}{2}=6$};
  \end{tikzpicture}
  \caption{Example of a set giving a lower bound for $n=12$}
  \label{figure:lowerBound}
\end{figure}
Let us compute the time it takes to span $[n]^2$ this way:
\begin{enumerate}
\item Phase 1 takes time $\floor{\frac{3(s-1)}{2}} \geq \frac{n}{2}-7$,
\item Phase 2 takes time $s > \frac{n}{3}-3$,
\item Phase 3 takes time
\[
 \sum_{i=0}^{\frac{n-s-3}{2}} (s+5 + 3i) = \frac{3n^2-2sn-s^2+8n-12s-11}{8} > \frac{5n^2}{18}+\frac{n}{2}+7,
\]
\item Phase 4 takes time
\[
 \frac{n+s-5}{6}(2n-1) = \frac{2n^2-11n+2ns-s+5}{6} > \frac{8n^2}{18}-\frac{26n}{9}+\frac{4}{3}.
\]
\end{enumerate}
Therefore, this way of infecting $[n]^2$ takes time at least $\frac{13n^2}{18}-\frac{14n}{9} -\frac{5}{3}$ to complete and the lower bound on $M(n)$ is proved.

To find an upper bound on $M(n)$, we would like to improve Proposition~\ref{prop:step3} and show that there is a scheme of the form $(s_0, 2, [1]^{\leq 1}[4]^*[6]^*)$. The main issue is that, due to the cycle $4 \to 7 \to 5 \to 6 \to 4$ in Figure \ref{figure:moves}, there is no obvious way to order Move $4$s, $5$s, $6$s and $7$s in our schemes. Another problem we would have to face is the fact that divisibility constraints restrict the number of times we can apply particular moves to eventually construct the $n \times n$ square.

To deal with both issues we shall introduce a more general and rather abstract process in which fractional Moves $4$, $5$, $6$ and $7$ can be applied. In this process, our aim is also to infect the square $[n]^2$. It will be obvious that the maximum spanning time in this new process is at least as big as in the $2$-neighbour bootstrap percolation. To be more precise, we will allow the following fractional moves (recall Figure~\ref{figure:movesdirections}). For $x \in (0,\infty)$
\begin{enumerate}
\item Move $(4, x)$ applied to a rectangle $P \in \Rec(s,t)$ spans $P' \in \Rec(s+2x,t+x)$ in time $x(s+t+1)+3(x^2-x)/2$.
\item Move $(5, x)$ applied to a rectangle $P \in \Rec(s,t)$ spans $P' \in \Rec(s + x,t+2x)$ in time $x(s+t+1)+3(x^2-x)/2$.
\item Move $(6, x)$ applied to a rectangle $P \in \Rec(s,t)$ spans $P' \in \Rec(s,t+3x)$ in time $x(2s-1)$.
\item Move $(7, x)$ applied to a rectangle $P \in \Rec(s,t)$ spans $P' \in \Rec(s+3x,t)$ in time $x(2t-1)$.
\end{enumerate}

We note that the amount of time that each fractional move takes was chosen so that: (a) for fixed $i \in \{4, 5, 6, 7\}$ and positive real numbers $x$, $y$, applying Move $(i,x)$ to a rectangle $R$ to get some rectangle $R'$ and then applying Move $(i,y)$ to $R'$, is equivalent to applying Move $(i, x+y)$ to $R$ only; (b) when $x$ is a natural number then applying Move $(i,x)$ is equivalent to applying the original Move $i$ exactly $x$ times. Crucially, using the new fractional moves we shall be able to get rid of Move $5$ completely so that the remaining moves will be easy to order.

Let us note that although the first part of Section \ref{sec:valueMnn} could be seen as a significant extension of the methods developed in \cite{benevidesprzykucki02}, the idea of fractional moves is a new concept that has not been studied before.

Let $Q = (s_0, 2, m_1m_2 \ldots m_r)$ be a scheme solving $M(n)$ of the form
\[
 (s_0, 2, [1]^{\leq 1}[3]^{\leq 2}[4|5|6|7]^*) \mbox{ or } (s_0, 2, [3]^{\leq 2}[1]^{\leq 1}[2]^*[4|5|6|7]^*),
\]
that exists by Proposition~\ref{prop:step3}. Let $A$ be an $(n,n)$-perfect set determined by $Q$ and let $P_0 \subset P_1 \subset \ldots \subset P_r \in \Rec(n,n)$ be the sequence of rectangles associated with it with $P_i \in \Rec(s_i,t_i)$. Let $j_0$ be such that $P_{j_0}$ is the rectangle obtained after the last occurrence of any of the Move $1$s, $2$s or $3$s. If there are no such moves, we set $j_0 = 0$. Since Move~$1$ is applied at most once and Move~$3$ at most twice we have $t_{j_0} \le 7$. Hence there is a scheme in which we first infect a rectangle $R \in \Rec(s_{j_0}, t_{j_0})$ where $t_{j_0} \le 7$ and then apply only Move $4$s, $5$s, $6$s or $7$s. Without loss of generality assume that $s_{j_0} \geq t_{j_0}$.

Using (fractional) moves we shall first construct a particular triple
\[
Q'=(s_0, 2, m_1\ldots m_{j_0} (m'_{j_0+1},x_{j_0+1})(m'_{j_0+2},x_{j_0+2}) \ldots (m'_{r'},x_{r'}))
\]
that infects $[n]^2$ in our generalized process in time at least as big as $Q$ does in bootstrap percolation, and then bound from above the time it takes to execute $Q'$. Recall that by using Move $m_i$ in $Q$ we finish infection of a rectangle $P_i \in \Rec(s_i,t_i)$. We build $Q'$ using the following procedure in which our aim is to ensure that at each step $j \geq j_0$ the rectangles $P'_j \in \Rec(s'_j,t'_j)$ that we obtain in the generalized process satisfy $s'_j \geq t'_j$. This allows us to eliminate all occurrences of Move $5$ (for an example of this procedure see Figure \ref{figure:generalizedProcess}). Set $h=j_0+1$ and for $i=j_0+1, j_0+2, \ldots, r$ let:
\begin{enumerate}
  \item If $m_{i} = 4$ or $m_{i} = 7$ put $m'_{h} = m_{i}$, $x_{h}=1$ and increase $h$ by $1$.
  \item If $m_{i} = 6$ and $s_{i} \geq t_{i}$ put $m'_{h} = 6$, $x_{h}=1$ and increase $h$ by $1$.
  \item If $m_{i} = 5$ and $s_{i} \geq t_{i}$ put $m'_{h} = 4$, $m'_{h+1} = 6$, $x_{h}=x_{h+1}=1/2$ and increase $h$ by $2$; note that in the generalized process this pair of fractional moves takes time \[
    \frac{s_{i-1}+t_{i-1}+1}{2}-\frac{3}{8} +\frac{2(s_{i-1}+1)-1}{2} = \frac{3s_{i-1}+t_{i-1}}{2}+\frac{5}{8},                                                                                                                                                                                               \]
 while the original Move $5$ takes $s_{i-1}+t_{i-1}+1$ steps which is less than the former value as we must have $s_{i-1} \geq t_{i-1}+1$.
  \item If $m_{i} = 5$ or $m_{i} = 6$ and $s_{i-1} = t_{i-1}$ then
  \begin{itemize}
    \item redefine $Q$ by, for $i \leq \ell \leq r$, changing each $m_{\ell}=4$ to $5$, $m_{\ell}=5$ to $4$, $m_{\ell}=6$ to $7$ and $m_{\ell}=7$ to $6$,
    \item note that after this ``mirror reflection'' the spanning time of $Q$ does not change,
    \item since now $m_{i} = 4$ or $m_{i} = 7$ then, like in case 1, put $m'_{h} = m_{i}$, $x_{h}=1$ and increase $h$ by $1$.
  \end{itemize}
  \item If $m_{i} = 6$ and $s_{i-1} = t_{i-1}+2$ (and hence $s_{i} = t_{i}-1$) then
  \begin{itemize}
    \item redefine $Q$ by setting $m_{i}=5$ so that $s_{i} = t_{i}+1$ and, for $i+1 \leq \ell \leq r$, by changing each $m_{\ell}=4$ to $5$, $m_{\ell}=5$ to $4$, $m_{\ell}=6$ to $7$ and $m_{\ell}=7$ to $6$,
    \item note that both the new and the old Move $m_{i}$ take $2s_{i-1}-1$ time steps and that after this modification $Q$ still spans $[n]^2$ in maximum time,
    \item put $m'_{h} = 4$, $m'_{h+1} = 6$, $x_{h}=x_{h+1}=1/2$ and increase $h$ by $2$; note that, like in case 3, in the generalized process this pair of fractional moves takes strictly more steps than the original Move $5$.
  \end{itemize}
  \item Finally we show that the only missing case $m_{i} = 6$, $s_{i-1} = t_{i-1}+1$ and $s_{i} = t_{i}-2$ cannot occur: if it did then we could increase the spanning time of $Q$ by $1$ step (contradicting its maximality) by applying the following modifications:
  \begin{itemize}
    \item redefine $Q$ by setting $m_{i}=4$ and, for $i+1 \leq \ell \leq r$, by changing each $m_{\ell}=4$ to $5$, $m_{\ell}=5$ to $4$, $m_{\ell}=6$ to $7$ and $m_{\ell}=7$ to $6$,
    \item note that now $s_i = t_i +2$ and that after this ``mirror reflection'' $Q$ still spans $[n]^2$,
    \item new Move $m_{i}$ takes $s_{i-1}+t_{i-1}+1=2s_{i-1}$ time steps while the old Move $m_{i}$ took $2s_{i-1}-1$ time steps; further steps take the same time as before thus $Q$ could not be a scheme.
  \end{itemize}
\end{enumerate}

\begin{figure}[ht]
  \centering
  \begin{tikzpicture}[scale=\tikzscale]
    \fill[color=lightgray] (0,0) rectangle (2.5,2);
    \draw[step=.5cm,gray]
    (0,0) grid (7.5,7.5);
	\draw[->,gray] (0,0)--(0,8) node [left] {$\ell$};    
	\draw[->,gray] (0,0)--(8,0) node [below] {$k$};    
	    
    \draw[line width=0.1mm,black,dashed]
    (0,0) -- (7.5,7.5);
    \tikzstyle{every node}=[draw,shape=circle,fill=black,inner sep=0pt,minimum size=3pt]
    \draw[line width=0.3mm,black]
    (2.5,1) node{} -- (2.5,2) node{} -- (3.5,2.5) node{} -- (3.5,4) node{} -- (4,5) node{} -- (5,5.5) node{} -- (6,6) node{} -- (6,7.5) node{} -- (7.5,7.5) node{};
    \draw[line width=0.5mm,black,dashed]
    (2.5,1) -- (2.5,2) -- (3.5,2.5) -- (4,2.75) node{} -- (4,3.5) node{} -- (5,4) node{} -- (5.5,4.25) node{} -- (5.5,5) node{} -- (6,5.25) node{} -- (6,6) -- (7.5,6) node{} -- (7.5,7.5);
    \tikzstyle{every node}=[]
    \draw (3.75,-0.25)  node [below] {$s_i$};
    \draw (-0.25,3.75) node [left] {$t_i$};
  \end{tikzpicture}
  \caption{Example of the generalized infection process for $n=15$ (which is not a scheme
for M(15), but which we use for demonstration purposes). Circular marks depict dimensions of rectangles $P_i \in \Rec(s_i,t_i)$ and $P'_i \in \Rec(s'_i,t'_i)$ obtained after consecutive moves. In this example we have a triple (which is not a scheme for $M(15)$ but we use it for demonstration purpose) $Q=(5,2,34654467)$ (solid line) and its modification $Q'=(2,5,3(4,1)(4,1/2)(6,1/2)(4,1)(4,1/2)(6,1/2)(4,1/2)$ $(6,1/2)(7,1)(6,1))$ (dashed line); note that here $j_0=1$, $s_{j_0}=5$ and $t_{j_0}=4$ (shaded rectangle represents the rectangle $P_{j_0}$).}
  \label{figure:generalizedProcess}
\end{figure}
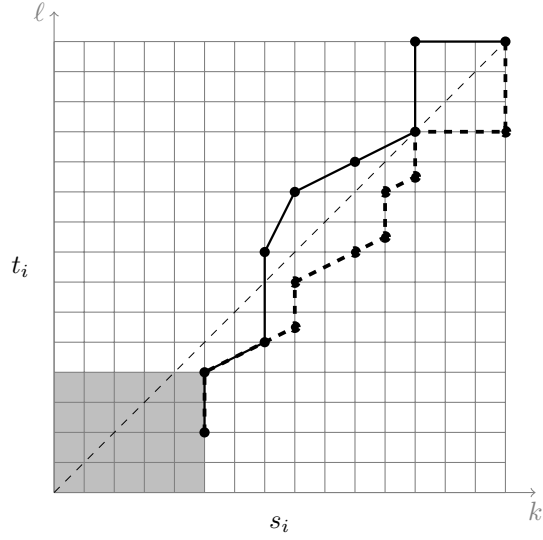

We do not have any occurrences of Move $5$ in $Q'$ and Move $4$s, $6$s and $7$s occur in multiples of $1/2$, i.e., all $x_i$'s are either $1/2$ or $1$. In Table \ref{table4} we show that wanting to maximize infection time we should keep the order of half--moves as suggested in Figure \ref{figure:moves}. That is, we should have Move $7$s followed by $4$s and finally by $6$s.

\begin{table}[ht]\centering
  \begin{tabular}{|c||c|c|c|} 
    \hline
       & $m'_i = 4$ & $m'_i = 6$ & $m'_i = 7$ \\
    \hline
       & & & \\[-1em]
    \hline
    $m'_{i-1} = 4$ & $k+\ell-2$ & $(3k+\ell)/2-15/8$ & $(k+3\ell)/2-15/8$ \\
    \hline
    $m'_{i-1} = 6$ & $(3k+\ell)/2-17/8$ & $2k-1$ & $k+\ell-5/2$ \\
    \hline
    $m'_{i-1} = 7$ & $(k+3\ell)/2-13/8$ & $k+\ell-5/2$ & $2\ell-1$ \\
    \hline
  \end{tabular}
  \caption{Time taken by consecutive half--Moves $(m'_{i-1},1/2)(m'_{i},1/2)$, assuming that $P'_i \in \Rec(k, \ell)$.}  
  \label{table4} 
\end{table}
Thus, for some $x,y,z \in [0, \infty)$, we obtain $Q'' =(s_0, 2, m_1 \ldots m_{j_0}(7,x)(4,y)(6,z))$ that takes at least as long to infect $[n]^2$ in our generalized infection process as a scheme $Q$ solving $M(n)$ does in bootstrap percolation. Denote the rectangle that we obtain when we apply Move $(7,x)$ to $P_{j_0}$ by $P_{j_0+x} \in \Rec(s,t)$ and note that we must have $y = (n-s)/2$ and $z = (n - t - (n-s)/2)/3$. Recall that $P_{j_0} \in \Rec(s_{j_0}, t_{j_0})$ with $t_{j_0} \leq 7$ and therefore $s = s_{j_0} + 3x$ and $t = t_{j_0} \leq 7$. To bound the spanning time of $Q''$ from above we may start by being generous and saying that $M(P_{j_0+x}) \le st \le 7s$. We then compute the time needed to apply Move $(4,y)$ and Move $(6,z)$. We conclude that the percolation time of $Q''$ can be bounded from above by

\begin{multline*}
s t + \frac{(n-s)}{2}(t+s+1) +\frac{3}{2}\frac{(n-s)}{2}\frac{(n-s-2)}{2} + \frac{(n-t-\frac{n-s}{2})}{3}(2n-1) \leq \\
\leq 7s + \frac{(n-s)(s+8)}{2} + \frac{3(n-s)(n-s-2)}{8} + \frac{(n+s)(2n-1)}{6} = f_n(s).
\end{multline*}

Maximizing $f_n(s)$ over $0 \leq s \leq n$ we find that its maximum is $f(\frac{n+43}{3}) = \frac{13}{18}n^2 + \frac{77}{18}n + \frac{1849}{72}$. That gives an upper bound on $M(n)$ and therefore completes the proof of Theorem \ref{theorem:arbitraryAsymptotics}. \endproof

\begin{remark}
We believe that $M(n)$ is achieved using a scheme of the form $(s, 2, [1]^{\le 1}[4]^*[6]^*)$ where $s = n/3 + O(1)$. However, we expect that a much more tedious case analysis might be necessary to prove this statement.
\end{remark}

\section{Concluding remarks and further questions}
\label{sec:questions}

In this paper we give the asymptotic formula for the maximum percolation time in the grid $[n]^2$ under $2$-neighbour bootstrap percolation. Our results allow us to prove the following two theorems about the maximum time of $2$-neighbour bootstrap percolation in other related graphs.

\begin{theorem}
\label{theorem:torus}
 Let $\TT_n^2$ be the $2$-dimensional $n\times n$ discrete torus and let $M^{\TT}(n)$ denote the maximum percolation time in $\TT_n^2$. Then $M^{\TT}(n) = 13n^2/18 + O(n)$.
\end{theorem}
{\par{\it Sketch of Proof}. \ignorespaces}
 To see that $M^{\TT}(n) \geq 13n^2/18 + O(n)$ consider, for $n \geq 4$, an $(n-2,n-2)$-perfect set $A$ that spans the square $(n-2) \times (n-2)$ in time $M(n-2) = 13n^2/18 + O(n)$ and such that $I_A(n-2,n-2) = M(n-2)$. Then the set $A \cup \{(n-1,n-1)\}$ percolates $\TT_n^2$ in time at least $M(n-2)$.

 For the upper bound on $M^{\TT}(n)$ it is not enough to argue that $\TT_n^2$ having greater connectivity than the square $n \times n$ implies that the infection process in $\TT_n^2$ runs faster. This is because there exists sets that percolate $M^{\TT}(n)$ but do not percolate the square $n \times n$, e.g., a diagonal minus one site. However, $2$-neighbour bootstrap percolation on $\TT_n^2$ can be seen as a similar ``rectangle process'' as described in Proposition \ref{prop:rectangles}. Performing a case analysis like in the proof of Theorem \ref{theorem:Mnrecurrence} with a bit of extra care needed to accommodate for the effect of ``folding'' it follows that $M^{\TT}(n) \leq 13n^2/18 + O(n)$.
\endproof

Using an asymmetric version of Proposition \ref{prop:step3} and the idea of fractional moves, for $\alpha \in (0,1)$ and $n$ large, assuming that $\alpha n$ is a natural number, we can determine the asymptotic value of $M(n,\alpha n)$. All we need to do is, for both $P_0 \in \Rec(s_0,2)$ and $P_0 \in \Rec(2,s_0)$, to follow the same reasoning as in Section \ref{sec:valueMnn} when we obtained the upper bound on $M(n)$. Constructing a set percolating in $[n] \times [\alpha n]$ in an essentially maximum time, and hence obtaining a corresponding lower bound on $M(n,\alpha n)$, is then immediate.

\begin{theorem}
 \label{theorem:maxRectangles}
We have:
\begin{enumerate}
 \item If $\frac{1}{3} \leq \alpha \leq 1$ then $M(n,\alpha n) = \left(\frac{2\alpha}{3}+\frac{1}{18}\right)n^2 + O(n)$.
 \item If $0 < \alpha \leq\frac{1}{3}$ then $M(n,\alpha n) = \left(\alpha-\frac{\alpha^2}{2}\right)n^2 + O(n)$.
\end{enumerate}
\end{theorem}
{\par{\it Sketch of Proof}. \ignorespaces}
For $\frac{1}{3} \leq \alpha \leq 1$ a construction that gives us the lower bound on $M(n,\alpha n)$ first infects a roughly $\frac{n}{3} \times 3$ rectangle in time $O(n)$, then using Move 4 $\left(\frac{n}{3}+O(1)\right)$ times extends it to a roughly $n \times \frac{n}{3}$ one in time $\frac{5n^2}{18}+O(n)$, and then finishes the infection in additional $\left(\frac{2\alpha}{3}-\frac{2}{9}\right)n^2+O(n)$ time steps using Move~6 $\left(\left(\frac{\alpha}{3}-\frac{1}{9}\right)n+O(1)\right)$ times.

A construction that gives us the lower bound on $M(n,\alpha n)$ when $0 < \alpha \leq\frac{1}{3}$ first infects a roughly $(1-2\alpha) n \times 3$ rectangle in time $O(n)$, and then finishes the infection in additional $\left(\alpha-\frac{\alpha^2}{2}\right)n^2 + O(n)$ time steps using Move 4 $\left(\alpha n+O(1)\right)$ times.
\endproof

The most obvious continuation and generalization of our work is establishing the maximum percolation time in the grid $[n]^d$ under $r$-neighbour bootstrap percolation, for all values of $d$ and $r$. Let us present the following partial result for $r=2$.

Let $T^d(A)$ denote the time that $A$ takes to percolate in $[n]^d$ under $2$-neighbour bootstrap percolation, so that $T(A) = T^2(A)$. Then the maximum percolation time in $2$-neighbour bootstrap percolation in $[n]^d$ is defined as
\[
M^d(n) = \max \{ T^d(A): \closure{A}=[n]^d \},
\]
so that $M(n) = M^2(n)$. Together with Simon Griffiths from the Instituto Nacional de Matem{\'a}tica Pura e Aplicada, Rio de Janeiro, Brazil, we proved the following theorem.
\begin{theorem}
\label{theorem:n^2}
For all $d \geq 1$ fixed,
 \[
 \frac{6d^2-5d-1}{18}n^2 + O(n) \leq M^d(n) \leq \frac{d^2}{2}n^2 + O(n).
 \]
\end{theorem}
{\par{\it Sketch of Proof}. \ignorespaces}
 For the lower bound, we generalize the construction in Theorem~\ref{theorem:arbitraryAsymptotics} to all dimensions. We first show by induction that for all $d \geq 1$ we can infect a $d$-dimensional cuboid $C_d$ with sides of length $n \times \ldots \times n \times \lfloor n/3 \rfloor$ in time $t_d(n) = (6d^2-13d+7)n^2/18 + O(n)$. Note that for $d=1$ this quantity is $O(n)$ and for $d=2$ it is $5n^2/18 + O(n)$, agreeing with the description of Phase 3 in the proof of the lower bound in Theorem \ref{theorem:arbitraryAsymptotics}. Having infected a $d$-dimensional cuboid of that form we infect a $(d+1)$-dimensional one by generalizing our earlier construction (note that $C_d$ can be in fact seen as a $(d+1)$-dimensional cuboid with the $(d+1)$th coordinate equal to $1$): using $n/3$ times an equivalent of Move 4 we repetitively grow $C_d$ by $2$ in the $d$th dimension and by $1$ in the $(d+1)$th dimension (with the $i$-th application of the equivalent of Move 4 taking approximately $(2(d-1)+1/3)n+3(i-1)$ time steps 
these sum up to $(2d/3-7/18)n^2+O(n)$ steps). Thus the formula for $t_d$ follows by induction.

 Finally, having infected $C_d$ in time $t_d(n)$, we finish the infection of $[n]^d$ by performing $2n/9$ times an equivalent of Move 6. Namely, we repetitively grow $C_d$ by $3$ in the $d$th direction, each such operation taking $2(d-1)n+O(1)$ time steps. Thus the lower bound on $M^d(n)$ follows.

 For the upper bound we first use a generalization of Proposition \ref{prop:rectangles} to all $d \geq 2$ (see Lemma 2.3 in~\cite{bootstraphigh}). Namely, if a cuboid $C$ in $[n]^d$ is internally spanned by a set $A$ then there exist some two strictly smaller cuboids in $C$ that are internally spanned by two disjoint subsets of $A$ and the union of which internally spans $C$. Reapplying this proposition inductively we see that the maximum percolation time in $[n]^d$ is bounded from above by the sum of percolation times of strictly smaller and smaller cuboids internally spanned by unions of two fully infected cuboids.

 We then show that for any two fully infected cuboids $C_1, C_2$ such that $\closure{C_1 \cup C_2} = C$ we have that $C_1 \cup C_2$ internally spans $C$ in time not larger than $\diam(C)+1$. We prove this fact by first noticing that (unless $C_1 \cup C_2 = C$ in which case the fact is trivial) there must exist some $v \in C$ such that $\dist (v,C_1) = \dist (v,C_2) = 1$. Obviously $v$ becomes infected at the first step of the process. We then prove by induction on $t \geq 2$ that all sites in $C$ at distance $t-1$ from $v$ are infected after at most $t$ steps of the process. This implies that $\diam(C)+1$ is an upper bound on the time that $C_1 \cup C_2$ takes to percolate $C$. Thus we obtain an upper bound on $M^d(n)$ equal to $\sum_{i=1}^{d(n-1)+1}i = d^2 n^2 / 2 + O(n)$.
\endproof

Note that the lower bound on $M^d(n)$ is sharp for $d=2$, i.e., it gives the right constant $13/18$. We believe that it is in fact sharp for all $d \geq 1$ and this motivates the following conjecture.
\begin{conj}[Benevides, Griffiths, Przykucki]
For all $d \geq 1$ fixed,
\[
 M^d(n) = \frac{6d^2-5d-1}{18}n^2 + O(n).
\]
\end{conj}

Another natural question which we leave for further work is the one about the maximum percolation time for higher infection thresholds in $[n]^d$. However, it is well known that percolation in $[n]^d$ for $r \geq 3$ is a completely different process from the one for $d=2$. For example, for $r \geq 3$, there is no description analogous to the rectangle process given in Proposition \ref{prop:rectangles}. It is therefore plausible that the maximum percolation time problem for higher infection thresholds will be completely different in nature. For example, one can expect a jump in the value of the maximum percolation time from $\Theta(n^2)$ for $r=2$ to $\Theta(n^d)$ for $r \geq 3$.
\begin{question}
What is the maximum percolation time in $r$-neighbour bootstrap percolation on $[n]^d$ for $r \geq 3$?
\end{question}

\appendix 

\section{Analysis of small cases} \label{appendix}

Assume that $(s_0, t_0, m_1m_2\ldots m_r)$ is a scheme for $M(k,\ell)$ for $k, \ell \geq 3$, $(k,\ell) \neq (3,3)$. Let $A$ be a $(k,\ell)$-perfect set described by it and let $P_0 \subset P_1 \subset \ldots \subset P_r \in \Rec (k,\ell)$ be the sequence of rectangles associated with $A$. We treat a number of small cases to exclude some, a priori possible, values for the numbers $s_0$ and $t_0$.

Suppose for a contradiction that $P_0 \in \Rec(s, 1)$. Since $P_1 \in \Rec(s_1,t_1)$ where $s_1, t_1 \geq 3$ and $\max\{s_1, t_1\} \geq 4$, one of the following cases must occur:
\begin{enumerate}
\item $P_1 \in \Rec(s, 3)$ with $s \geq 4$: since we have $M(s-1,2) \geq 3$,  by applying Move~1 to $[s-1] \times [2]$ we see that $M(s,3) \geq (s-1)+3 = s+2$. However, for $P_0 \in \Rec(s, 1)$ and $P_1 \in \Rec(s, 3)$, as in the infection process defined by $A$, it takes time at most $s+1$ to infect all sites in $P_1$ since both ending sites of the rectangle $P_0$ must be initially infected. This contradicts the fact that at every step~$i$ the time that $A$ takes to percolate $P_i$ is maximum;

\item $P_1 \in \Rec(s+1, 3)$ with $s \geq 3$: since we have $M(s,2) \geq 3$, by applying Move~1 to $[s] \times [2]$ we see that $M(s+1,3) \geq s+3$. However, for $P_0 \in \Rec(s, 1)$ and $P_1 \in \Rec(s+1,3)$, as in the infection process defined by $A$, it takes time at most $s+2$ to infect all sites of $P_1$ (by the same argument as above). This again contradicts the fact that $A$ is $(n,n)$-perfect;

\item $P_1 \in \Rec(s, 4)$ with $s \geq 3$: since we have $M(s,2) \geq s$, by applying Move~3 to $[s] \times [2]$ we see that $M(s,4) \geq s+s+1 = 2s+1$. However, for $P_0 \in \Rec(s, 1)$ and $P_1 \in \Rec(s, 4)$, as in the infection process defined by $A$, using again the same argument it takes time at most $2s-1$ to infect all sites of $P_1$. This contradicts the fact that $A$ is $(n,n)$-perfect.
\end{enumerate}

Thus, we may assume that $P_0 \notin \Rec(s, 1)$. Analogously, we may assume that $P_0 \notin \Rec(1,t)$ . Suppose now that $P_0 \in \Rec(3, 3)$. Considering $P_1 \in \Rec(s_1,t_1)$ up to symmetries one of the following cases must occur:
\begin{enumerate}
\item $P_1 \in \Rec(6, 3)$: by applying Move~7 it takes time $5$ to infect $P_1$ after $P_0$ is fully infected. This procedure takes time at most $M(3)+5 = 9$ to infect $P_1$. However, by applying Move~1 to $[5] \times [2]$ we see that $M(6,3) \geq M(5,2)+5 = 6+5=11$; this contradicts the fact that $A$ is $(n,n)$-perfect;
\item $P_1 \in \Rec(5, 4)$: by applying Move~4 it takes time $7$ to infect $P_1$ after $P_0$ is fully infected. This procedure takes time at most $M(3)+7 = 11$ to infect $P_1$. However, by applying Move~3 to $[5] \times [2]$ we see that $M(5,4) \geq M(5,2)+6 = 6+6=12$; this contradicts the fact that $A$ is $(n,n)$-perfect;
\item $P_1 \in \Rec(4, 4)$: by applying Move~1 it takes time $3$ to infect $P_1$ after $P_0$ is fully infected. This procedure takes time at most $M(3)+3 = 7$ to infect $P_1$. However, by applying Move~3 to $[4] \times [2]$ we see that $M(4) \geq M(4,2)+5 = 9$; this contradicts the fact that $A$ is $(n,n)$-perfect;
\item $P_1 \in \Rec(5, 3)$: by applying Move~2 it takes time $4$ to infect $P_1$ after $P_0$ is fully infected. This procedure takes time at most $M(3)+4 = 8$ to infect $P_1$. By applying Move~1 to $[4] \times [2]$ we also take time $M(4,2)+4 = 8$. Although this does not contradict the $(n,n)$-perfectness of $A$, we can replace it by an $(n,n)$-perfect set $A'$ whose infection process starts with a $P_0' \in \Rec(4,2)$ and expands to $P_1$, so that $A'$ takes the same time to percolate in $[n]^2$ as $A$.
\end{enumerate}

Thus, we may assume that $P_0 \notin \Rec(3, 3)$ and we have $P_0 \in \Rec(s, 2) \cup \Rec(2, s)$ for some $s \geq 3$.

\bibliographystyle{siam}

 \bibliography{mylargebib}
\end{document}